\documentclass[11pt,letterpaper]{article}
\usepackage{settings}

\title{Improved Nonnegativity Testing in the Bernstein Basis \\ via Geometric Means}
\author{Mitchell Tong Harris \\ \href{mitchh@mit.edu}{mitchh@mit.edu}  \and Pablo A. Parrilo \\ \href{parrilo@mit.edu}{parrilo@mit.edu}}

\begin{document}

\maketitle

\makeatletter
\def\blfootnote{\gdef\@thefnmark{}\@footnotetext}
\makeatother
\blfootnote{The authors are with the Laboratory for Information and Decision Systems (LIDS), Massachusetts Institute of Technology, Cambridge MA 02139. The first author was supported by the National Science Foundation Graduate Research Fellowship under Grant No. 2141064.}

\begin{abstract}
We develop a new kind of nonnegativity certificate for univariate polynomials on an interval. In many applications, nonnegative Bernstein coefficients are often used as a simple way of certifying polynomial nonnegativity. Our proposed condition is instead an explicit lower bound for each Bernstein coefficient in terms of the geometric mean of its adjacent coefficients, which is provably less restrictive than the usual test based on nonnegative coefficients. We generalize to matrix-valued polynomials of arbitrary degree, and we provide numerical experiments suggesting the practical benefits of this condition. The techniques for constructing this inexpensive certificate could potentially be applied to other semialgebraic feasibility problems.
\end{abstract}

\section{Introduction}
\label{sec:intro}
The cubic Bernstein polynomials (e.g., \cite{farouki2012bernstein, lorentz2012bernstein}) are 
\[ b_0(x) = (1-x)^3, \qquad b_1(x) = 3x(1-x)^2, \qquad b_2(x) = 3x^2(1-x), \qquad b_3(x) = x^3.
\]
Suppose we have a polynomial $p(x)$ such that
\begin{equation}\label{eqn:cubic}
    p(x) = p_0 b_0(x) + p_1b_1(x) + p_2b_2(x) + p_3b_3(x).
\end{equation}
We would like to find explicit conditions on the real numbers $p_0, p_1,p_2,p_3$ (called Bernstein coefficients) that guarantee $p(x) \geq 0$ on $[0,1]$. This task and its higher degree variants discussed in Section~\ref{sec:generalization} are central questions in applied mathematics. Research on nonnegativity certificates of different kinds is a classical topic in real algebraic geometry, including well-known work by Sturm, Hilbert, Artin, and others; see e.g.~\cite{BPRbook, powers2021certificates} and the references therein. More recently, concrete connections to applications in statistics, control, and optimization have heightened interest in finding computation-friendly nonnegativity certificates \cite{parrilo2000structured, papp2011optimization,ahmadi2016some}.
Existing conditions that have been used in applications as nonnegativity certificates of \eqref{eqn:cubic} 
 on intervals can roughly be classified as follows:
\begin{enumerate}
\item \textbf{Exact characterizations:} The Markov–Luk\'acs Theorem (\cite[p.\ 4]{szeg1939orthogonal}) gives a necessary and sufficient condition for nonnegativity on an interval. It states that a polynomial $p$ is nonnegative on $[0,1]$ if and only if there exist polynomials $s_1$ and $s_2$ such that
\[p(x) = x \, s_1(x)^2 + (1-x) \, s_2(x)^2.\] 
In the cubic case, the polynomials $s_1$ and $s_2$ are linear, so writing this condition in the language of Bernstein polynomials along with a characterization of polynomials that are sums of squares (SOS) yields the second-order cone (SOCP) condition in Lemma~\ref{lem:nonnegative_socp}. While this condition has fixed computational cost and is necessary and sufficient for a polynomial to be nonnegative on the interval, it is nontrivial to verify -- one has to search for the linear polynomials~$s_1$ and~$s_2$, typically via convex optimization.

Alternatively, an exact characterization can be given in terms of the discriminant of $p$. We elaborate on this approach in Section~\ref{sec:cubic-exact}. As mentioned in \cite{schmidt1988positivity}, conditions of this type on $p$ are typically more complicated to understand and to compute, particularly for higher degrees.

\item \textbf{Sufficient conditions:} A simple, sufficient condition for nonnegativity of the polynomial $p$ is $p_0,p_1,p_2,p_3 \geq 0$. Because each of the Bernstein polynomials is nonnegative on the interval, a nonnegative combination of them is nonnegative too. This idea was mentioned by Bernstein~\cite{bernshtein1915representation} in 1915, expanded in 1966 by~\cite{cargo1966bernstein}, and further discussed more recently in~\cite{lin1995methods}. The downside to this condition is that it is potentially conservative -- there are plenty of polynomials nonnegative on $[0,1]$ that have at least one negative Bernstein coefficient.
\end{enumerate}

We give names to the sets of polynomials that satisfy each of these conditions.
\begin{defn} $\mathcal{P}$ (``Positive'') is the set of cubic polynomials such that $p(x) \geq 0$ on $[0,1]$.
\end{defn}
\begin{defn} $\mathcal{NB}$ (``Nonnegative Bernstein'') is the set of cubic polynomials whose Bernstein coefficients are all nonnegative.
\end{defn} 

These conditions reflect two extremes in the tradeoff space between cost and quality of approximation. On the one hand, $\mathcal{P}$ is the full set, but testing membership is nontrivial. On the other hand, $\mathcal{NB}$ is a smaller set for which checking membership is easy. The purpose of this paper is to introduce a set that lies in between $\mathcal{NB}$ and $\mathcal{P}$, a set that better trades off accuracy for computational cost. In doing so, we highlight a technique for generating new sufficient conditions for nonnegativity based on making an explicit data-dependent choice of feasible solutions in the SOS characterizations.

In the cubic case, these conditions become the explicit set of inequalities $p_0, p_3 \geq 0$ and \begin{equation}
    p_2 \geq -2\sqrt{\frac{p_3\max(p_1,0)}{3}} \qquad \text{and} \qquad p_1 \geq -2\sqrt{\frac{p_0\max(p_2,0)}{3}}.
    \end{equation}
These inequalities imply nonnegativity of $p$ on the interval without requiring solving a program with additional decision variables. 
This new condition, and its higher degree generalizations, more gracefully balance the tradeoffs between efficiency and approximation power. We show how these fast-to-check, explicit conditions naturally generalize to arbitrary degree (Section 4) and polynomial matrices (Section 5); see Definition~\ref{def:GBmatrix} and Theorem~\ref{thm:arbitrary_general_inclusions}. Finally, we illustrate how using them in the right context may have practical benefits.

\subsection{Related work}

There is a long history of work related to nonnegativity certificates. Some of the oldest results are due to Descartes, Sylvester, Sturm, Budan and Fourier (see, for instance, \cite{bochnack2013real, BPRbook, anderson1998descartes}).
In the multivariate case, Hilbert's 1888 paper on sums of squares \cite{hilbert1888ueber} again revitalized the area. Ever since, there has been consistent work on characterizing large classes of nonnegative polynomials \cite{powers2011positive,powers2021certificates}. Due the intrinsic computational difficulty of the problem, to varying extents all these methods trade off conservativeness for computational ease. 

Bernstein's original 1915 work \cite{bernshtein1915representation} 
provided a sufficient condition for nonnegativity on the unit interval. This condition is exceptionally easy to test, and within about fifty years a higher dimensional version was adopted for applications in the context of computer aided design. The higher dimension version is the familiar Bézier curve \cite{de1963courbes, bezier1966definition, bezier1967definition}, which is a parametric curve such that each coordinate is a univariate Bernstein polynomial. Bernstein's sufficient condition in one dimension is exactly the defining condition of $\mathcal{NB}$; it can be generalized to the containment of a B\'ezier curve in the polytope defined by the convex hull of the Bernstein coefficient vectors~\cite{biran2018geometry}. This condition is tractable enough that it has been used in applications ranging from geometric constraints systems \cite{foufou2012bernstein} to problems of robust stability \cite{garloff2000application, garloff1997speeding,malan1992b,zettler1998robustness}. 

On the other hand, sum of squares nonnegativity certificates were not too practical until the development of further connections to tractable optimization problems. Sum of squares polynomials became useful for applications when Shor \cite{shor1987class}, Nesterov \cite{nesterov2000squared}, Parrilo \cite{parrilo2000structured} and Lasserre \cite{lasserre2001global} used convex optimization and semidefinite programming to algorithmically check this property. Sum of squares techniques allow a more precise tuning of the computational costs -- the more conservative one wants to be, the more expensive the method becomes. Further work built on this,  by restricting to subsets of the set of sum of squares for which checking membership is more tractable. One proposal is to look for easier-to-check certificates; e.g., instead of requiring the associated Gram matrix to be positive semidefinite, one can use the stronger condition of diagonal dominance or scaled diagonal dominance~\cite{ahmadi2019dsos}. 

Other classes of nonnegativity certificates have also been explored, with an eye towards potential computational benefits.  A thread of work that goes back to Reznick \cite{reznick1989forms}, uses a decomposition of a polynomial as a sum whose terms are nonnegative due to the AM-GM inequality. This idea, combined with relative entropy optimization~\cite{chandrasekaran2016relative} has been used to generate a variety of new nonnegativity conditions including those presented in \cite{fidalgo2011positive,chandrasekaran2016relative,iliman2016amoebas}. In general, these sets are incomparable with the set of sum of squares, so they exemplify another compromise between conservativeness and computational ease. These approaches have a similar flavor to our proposal; however, checking the membership condition of the existing methods still requires the solution of some program, so they are generally more expensive than checking Bernstein's condition. Because our conditions are explicit, and do not require solving any optimization programs, they sit closer to the cheaper, less explored end of the tradeoff spectrum. 

\section{Preliminaries} \label{sec:preliminaries-parent}
\subsection{Notation}
We use $\mathbb{R}_d[x]$ to denote the vector space of univariate polynomials of degree less than or equal to $d$. Given a set $S$, let $S^\circ$ be the interior of the set $S$.

While the paper begins with a focus on univariate polynomials, we eventually generalize our results to polynomial matrices in Section~\ref{sec:polynomial-matrices}. In that case, the analogue of being pointwise nonnegative is to be pointwise positive semidefinite. Let $\mathcal{P}_{d}^{n}$ denote the set of symmetric $n \times n$ univariate polynomial matrices of degree $d$ with nonnegative eigenvalues on the unit interval. 
For example, $\mathcal{P}_3^1$ is the set of cubic polynomials nonnegative on $[0,1]$. We use the same subscript/superscript convention for other sets such as in $\mathcal{NB}_3^1$.
If we drop the dimension superscript or the degree subscript it should be implied from context.

If $x \in \mathbb{R}$, let $x_+ = \max(0, x)$. Let $X, Y \in S^n$ be $n \times n$ symmetric matrices. Let $X_+$ be the orthogonal projection of $X$ onto the positive semidefinite cone. We write $X \succeq Y$ if $X-Y$ is positive semidefinite and $X \succ Y$ if $X-Y$ is positive definite. If $X$ is positive semidefinite, then $X^{\frac{1}{2}}$ is its matrix square root, which is the unique positive semidefinite matrix $Z$ such that $Z^2 = X$. Both the projection and the square root can be computed with eigenvalue decompositions; a review of the algorithms are given in Appendix~\ref{sec:preliminaries-proofs}.

\subsection{Second order cones and Bernstein polynomials}\label{sec:preliminaries}

The two main ingredients in our method are Bernstein polynomials and second order cones. We define both before presenting the second order cone condition for nonnegativity of a cubic Bernstein polynomial.

\begin{defn}
The three-dimensional \emph{rotated second order cone} $\mathcal{Q} \subset \mathbb{R}^3$ is the set
\begin{equation}\label{eqn:definingQ}
    \mathcal{Q} = \{ (x_0, x_1; x_2) \in \mathbb{R}^+ \times \mathbb{R}^+ \times \mathbb{R} \, | \, 2x_0 x_1 \geq x_2^2\}.
\end{equation}
\end{defn}
This is a closed convex cone. The point $(x_0, x_1; x_2) \in \mathcal{Q}$ if and only if the matrix \begin{equation}\label{eqn:2x2psdQversion} \pmat{x_0 & \frac{1}{\sqrt{2}}x_2 \\  \frac{1}{\sqrt{2}}x_2 & x_1} \end{equation} is positive semidefinite. The second order cone is self-dual; i.e., $\mathcal{Q}$ is the set of points whose inner product with every point in $\mathcal{Q}$ is nonnegative. In particular, if $(x_0, x_1; x_2) \in \mathcal{Q}$ and $(y_0, y_1; y_2) \in \mathcal{Q}$, then $x_0 y_0 + x_1 y_1 + x_2y_2 \geq 0$. See \cite{alizadeh2003second} for more background on~$\mathcal{Q}$.

\begin{defn}
Let $0 \leq i \leq d$. The $i$th degree $d$ \emph{Bernstein polynomial} is \begin{equation} b_i(x) = \binom{d}{i}x^{i}(1-x)^{d-i}.\end{equation}  The degree $d$ of $b_i$ will be made clear from context.
\end{defn}
There are two well-known properties of Bernstein polynomials we use:
\begin{enumerate}
    \item They form a basis: $\{b_i\}_{i=0}^d$ is a basis of $\mathbb{R}_d[x]$.
    \item The basis is nonnegative: For every $d$, $0 \leq i \leq d$, and $0 \leq x \leq 1$, $b_i(x) \geq 0$.
\end{enumerate}
Even though the second property is stated for the unit interval, Bernstein polynomials are useful in a more general context. If the interval of interest is not $[0,1]$ but rather $[r,s]$, change $b_i(x)$ to $b_i(\frac{x-r}{s-r})$. This change of variables allows us to restrict our attention to the $[0,1]$ case without loss of generality.

The following lemma is an algebraic identity relating consecutive Bernstein polynomials. The identity provides a connection between Bernstein polynomials and the second order cone $\mathcal{Q}$.

\begin{restatable}{lem}{bernsteinsocp}\label{lemma:bernstein_socp}
    Let \begin{equation}\label{eqn:midef}m_i = \frac12 \cdot \frac{(i+1) \cdot (d-i+1)}{i \cdot (d-i)}.\end{equation} Then for $1 \leq i \leq d-1$,
        \[2m_i \, b_{i-1}(x) b_{i+1}(x) = b_{i}(x)^2.\] 
        In particular, for every $0 \leq x \leq 1$,
\begin{equation}\label{eqn:bernstein_socp} \Bigg(b_{i-1}(x), b_{i+1}(x); \frac{1}{\sqrt{m_i}}b_i(x) \Bigg) \in \mathcal{Q}.\end{equation}
\end{restatable}
\begin{proof}
        We calculate
        \begin{equation}
        \begin{aligned}
            2m_ib_{i-1}(x)b_{i+1}(x) &= \frac{(i+1) \cdot (d-i+1)}{i \cdot (d-i)}\left(\binom{d}{i-1}x^{i-1}(1-x)^{d-i+1}\right) \left( \binom{d}{i+1}x^{i+1}(1-x)^{d-i-1}\right)\\
            &= \frac{(i+1) \cdot (d-i+1)}{i \cdot (d-i)} \frac{d!}{(i-1)!(d-i+1)!}\frac{d!}{(i+1)!(d-i-1)!}x^{2i}(1-x)^{2(d-i)}\\
            &= \left(\frac{d!}{i!(d-i)!}\right)^2x^{2i}(1-x)^{2(d-i)}\\
            &= b_{i}^2(x).
        \end{aligned}
        \end{equation}
        Therefore, the quadratic inequality required by \eqref{eqn:bernstein_socp} holds with equality -- geometrically, these points all lie on the boundary of $\mathcal{Q}$.
    \end{proof}

\begin{defn}
Let $p \in \mathbb{R}_d[x]$. The \emph{Bernstein coefficients} of $p$ are the unique real numbers $\{p_i\}_{i=0}^d$ such that
\begin{equation}
    p(x) =\sum_{i=0}^d p_i b_i( x).
\end{equation}
\end{defn}
\noindent We write $p_i$ to refer to the Bernstein coefficients of a polynomial $p$.

\paragraph{Subdivision method} The Bernstein coefficients are a key ingredient of the \textit{subdivision method} to prove nonnegativity of a polynomial on the interval. The basic idea is to recursively bisect the domain, until a given termination criterion is satisfied. The subdivision method with nonnegativity criterion $S$ is described by the following algorithm.
\begin{quote}
\begin{itemize}
    \item[Step 1.] Check if $p \in S$. If yes, terminate.
    \end{itemize}\end{quote}
\begin{quote}\begin{itemize}
    \item[Step 2.] Otherwise, subdivide $p$ into $p^1$ and $p^2$ (explained below). Repeat algorithm with $p^1$ and $p^2$ and wait until both calls terminate. 
\end{itemize}
\end{quote}
If the algorithm terminates, then $p \in \mathcal{P}$. Termination of the subdivision algorithm is guaranteed for positive polynomials when $S = \mathcal{NB}$ \cite{leroy2012convergence}. 
As explained below, if the polynomial $p$ is written in the Bernstein basis, each of the steps can be efficiently computed. 

Step 1 of the subdivision method is to check an explicit nonnegativity condition (which we interchangeably call a ``termination'' or ``nonnegativity'' criterion). If $S = \mathcal{NB}$, then we can efficiently check whether $p \in \mathcal{NB}$ when the Bernstein coefficients are given.

Step 2 of the subdivision method is to subdivide a polynomial. The idea of subdivision is to split the unit interval into two subintervals and give two polynomials that agree with the original on each half. Those new polynomials are rescaled to be defined on $[0,1]$. \begin{figure}[htbp]%
\begin{subfigure}{.25\textwidth}
    \resizebox{\textwidth}{!}{\begin{tikzpicture}
        \begin{axis}[
                xmin=0, xmax=1,
                ymin=0, ymax=1,
                xlabel=$x$,
                ylabel=$(1-2x)^2$,
                axis lines=middle,
                xtick={0,0.2,0.4,0.6,0.8,1},
                ytick={0,0.2,0.4,0.6,0.8,1},
                clip=false,
            ]
            \addplot[
                domain=0:1,
                samples=100,
                color=red,
                thick,
            ]
            {(1-2*x)^2};
            \addplot [mark=none, dashed] coordinates {(.5,-.1) (.5,1.1)};
        \end{axis}
    \end{tikzpicture}}
    \caption{$p(x)=(1-2x)^2$}
    \label{fig:subdiv-explanation-0}
\end{subfigure}
\hfill
\begin{subfigure}{.25\textwidth}
\resizebox{\textwidth}{!}{\begin{tikzpicture}
        \begin{axis}[
                xmin=0, xmax=1,
                ymin=0, ymax=1,
                xlabel=$x$,
                ylabel=$(1-x)^2$,
                axis lines=middle,
                clip=false,
            ]
            \addplot[
                domain=0:1,
                samples=100,
                color=red,
                thick,
            ]
            {(1-x)^2};
            \addplot [mark=none, dashed] coordinates {(1,-.1) (1,1.1)};
        \end{axis}
    \end{tikzpicture}}
    \caption{$p^1(x)=(1-x)^2$}
    \label{fig:subdiv-explanation-1}\end{subfigure}
\hfill 
\begin{subfigure}{.25\textwidth}\resizebox{\textwidth}{!}{\begin{tikzpicture}
        \begin{axis}[
                xmin=0, xmax=1,
                ymin=0, ymax=1,
                xlabel=$x$,
                ylabel=$x^2$,
                axis lines=middle,
                clip=false,
            ]
            \addplot[
                domain=0:1,
                samples=100,
                color=red,
                thick,
            ]
            {(x)^2};
            \addplot [mark=none, dashed] coordinates {(0,-.1) (0,1.1)};
        \end{axis}
    \end{tikzpicture}}
    \caption{$p^2(x)=x^2$}
    \label{fig:subdiv-explanation-2}
    \end{subfigure}
    \caption{Subdivision of $(1-2x)^2$ into the polynomials shown in Figures~\ref{fig:subdiv-explanation-1} and \ref{fig:subdiv-explanation-2} as described in \eqref{eqn:subdivision-def}.}
    \label{fig:subdiv-explanation}
    \end{figure}
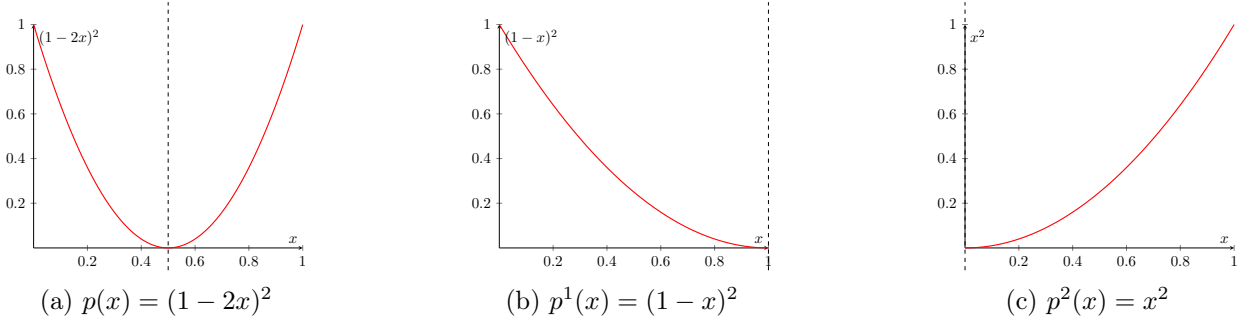 An illustration of this idea is given for the polynomial $(1-2x)^2$ in Figure~\ref{fig:subdiv-explanation}. Let $p(x) \in \mathbb{R}_d[x]$. Subdividing $p$ gives  polynomials $p^1(x)$ and $p^2(x)$ such that 
\begin{equation}\label{eqn:subdivision-def}
    p(x) = \begin{cases} p^1(2x) \qquad  &0 \leq x \leq \frac{1}{2}\\
     p^2(2x-1) \qquad &\frac{1}{2} \leq x \leq 1.\end{cases}
\end{equation}
When the Bernstein coefficients of a polynomial $p$ are known, De Casteljau's algorithm \cite{farouki2012bernstein} provides an efficient formula for computing the Bernstein coefficients of $p^1$ and $p^2$. For cubic polynomials, the Bernstein coefficients of $p^1$ and $p^2$ are \[\pmat{1 & 0 & 0 & 0\\ \frac{1}{2} & \frac{1}{2} & 0 & 0\\
\frac{1}{4} & \frac{2}{4} & \frac{1}{4} & 0 \\ \frac{1}{8} & \frac{3}{8} & \frac{3}{8} & \frac{1}{8} }\pmat{p_0\\p_1\\p_2\\p_3} \qquad \text{and} \qquad \pmat{\frac{1}{8} & \frac{3}{8} & \frac{3}{8} & \frac{1}{8}\\ 0 & \frac{1}{4} &\frac{2}{4} & \frac{1}{4}\\
0 & 0 & \frac{1}{2} & \frac{1}{2} \\ 0&0&0&1} \pmat{p_0\\p_1\\p_2\\p_3}.\]

 In Section~\ref{sec:numerical}, we explore how the number of subdivision steps required by this method is impacted by using our new nonnegativity criterion versus using $\mathcal{NB}$. %
 While we consider just one variant of univariate subdivision, the general idea of subdivision is useful for a variety of problems in robustness analysis and computer-aided design~\cite{lane1980theoretical}.

\section{The cubic case}
We begin with the cubic case because it is the simplest setting in which our results are nontrivial. First we present the exact condition for nonnegativity alluded to in Section~\ref{sec:intro} and describe a strategy to use it to get new conditions on polynomials that imply nonnegativity. Then we employ this method to provide one based on geometric means in Section~\ref{sec:P2_def} and justify why it is natural in Section~\ref{sec:why}. Since in the cubic case the set of nonnegative polynomials can be understood via the discriminant, Section~\ref{sec:cubic-exact} contains an exact discriminantal condition, unrelated to our geometric mean one, in order to compare how much simpler ours is. We finish this section by showing how the geometric mean condition can also be used to construct an exact nonnegativity characterization.

The following lemma, written in the language of Bernstein polynomials and second order cones, gives an exact characterization of cubic nonnegativity and is the basis for how we develop our new conditions. 

\begin{restatable}{lem}{nonnegativesocp}[Cubic nonnegativity and SOCP]\label{lem:nonnegative_socp} Let $p \in\mathbb{R}_3[x]$. The polynomial $p \in \mathcal{P}$ if and only if there exist $c_1, c_2 \in \mathbb{R}$ such that
\begin{equation}\label{eqn:nonnegative_socp}
    \Big(p_0, \; p_2-\frac{c_2}{3} ; \;  \frac{c_1}{\sqrt{6}} \Big) \in \mathcal{Q} \qquad \text{ and } \qquad 
    \Big(p_1-\frac{c_1}{3},  \; p_3;  \; \frac{c_2 }{\sqrt{6}}\Big)\in \mathcal{Q}.
\end{equation}
\end{restatable}
\begin{proof}First we prove sufficiency: If the second order cone conditions hold for $p$, then $p \in \mathcal{P}$. In this proof, we write $b_i$ to mean $b_i(x)$ for any fixed $0 \leq x \leq 1$. Plugging $n = 3$ into Lemma~\ref{lemma:bernstein_socp}, we get $(b_0, b_2;\sqrt{\frac{2}{3}} b_1) \in \mathcal{Q}$ and $(b_1, b_3; \sqrt{\frac{2}{3}}b_2) \in \mathcal{Q}$. 
Applying self duality of the second order cone, we get 
    \begin{equation}\begin{aligned}(p_0, p_2-\frac{c_2}{3}; \frac{c_1}{\sqrt{6}}) \cdot (b_0, b_2;\sqrt{\frac{2}{3}} b_1)& 
    = &-\frac{c_2}{3}b_2 + &\frac{c_1}{3}b_1 + p_0b_0 + p_2b_2 &\geq 0 \\ (p_1-\frac{c_1}{3}, p_3;  \frac{c_2}{\sqrt{6}}) \cdot (b_1, b_3; \sqrt{\frac{2}{3}}b_2)&
    = &\frac{c_2}{3}b_2 -&\frac{c_1}{3}b_1 + p_1b_1 + p_3b_3 &\geq 0.\end{aligned}\end{equation} Adding these, we get $p_0b_0 + p_1b_1 + p_2b_2 + p_3b_3 \geq 0$ as desired. That proves sufficiency.
    
    Next we prove necessity. If $p \geq 0$ on $[0,1]$, then by the Markov–Luk\'acs Theorem (\cite[p.\ 4]{szeg1939orthogonal}), $p(x) = xq_1^2(x) + (1-x) q_2^2(x)$ where $q_1(x)$ and $q_2(x)$ are linear functions. A basis of linear functions is given by $\{x, 1-x\}$, so there exist $M_1, M_2 \succeq 0$ such that 
    \begin{equation}
        p(x) = x \pmat{x \\ 1-x}^T M_1 \pmat{x \\ 1-x} + (1-x)\pmat{x \\ 1-x}^T M_2 \pmat{x \\ 1-x}.
    \end{equation}
    The affine constraints on the coefficients give $M_1 = \pmat{p_0 & \frac{c_1}{2} \\ \frac{c_1}{2} & 3p_2 - c_2}$ and $M_2 = \pmat{3p_1-c_1 & \frac{c_2}{2} \\ \frac{c_2}{2} &p_3}$ for some $c_1, c_2\in \mathbb{R}$. Multiplying $M_1$ by $\pmat{1 & 0 \\ 0 &\frac{1}{\sqrt{3}}}$ on both sides and $M_2$ by $\pmat{\frac{1}{\sqrt{3}} & 0 \\ 0 &1}$ on both sides gives two matrices that are positive semidefinite if and only if \eqref{eqn:nonnegative_socp} is satisfied.
    \end{proof}

\begin{remark}
Lemma~\ref{lem:nonnegative_socp} above can be interpreted as a specialization to the cubic case of the usual SDP characterization of nonnegative polynomials, adapted to the interval; see e.g., \cite[Lemma 3.33]{BPRbook}. This is because SDPs of size $2 \times 2$ can be equivalently written as second-order cone programs.
\end{remark}

\subsection{A new strategy} \label{sec:P2_def} 

 If we want to use Lemma~\ref{lem:nonnegative_socp} to prove that a given polynomial $p \in \mathcal{P}$ is nonnegative on the interval, we must find $c_1$ and $c_2$ that satisfy \eqref{eqn:nonnegative_socp}. Here are three possible strategies to find $c_1$ and $c_2$.

\paragraph{Convex optimization} Given $p$, one option is to solve the convex feasibility problem~\eqref{eqn:nonnegative_socp} for some feasible $c_1$ and $c_2$. This is a foolproof way, in the sense that for every $p \in \mathcal{P}$, there always exist (generally nonunique) solutions $c_i$. The drawback of this strategy is the computational cost.

\paragraph{Constant guess, independent of $p$} On the other extreme, one could ignore $p$ and always try the same constants $c_1$ and $c_2$. 
Given any constants $c_1$ and $c_2$ there is a set of feasible $p$ to \eqref{eqn:nonnegative_socp} induced by any such choice; this set of feasible $p$ is a (possibly empty) convex subset of $\mathcal{P}$. Given $p$ and the fixed constant choices $c_1$ and $c_2$, checking feasibility of~\eqref{eqn:nonnegative_socp} is as simple as verifying a few inequalities. 

In fact, this is the strategy used to construct $\mathcal{NB}$. The set of polynomials that are feasible in \eqref{eqn:nonnegative_socp} when $c_1 = c_2 = 0$ is exactly the set $\mathcal{NB}$. This also points out the main weakness of this strategy: there are plenty of polynomials in $\mathcal{P}$ that are not feasible in \eqref{eqn:nonnegative_socp} for $c_1 = c_2 = 0$. For example, consider $p = (1-x)(1-4x)^2 + x^3$, whose Bernstein coefficients are $(1,-2,3,1)$, not all of which are nonnegative. On the other hand, the conditions in Lemma~\ref{lem:nonnegative_socp} hold for $c_1 = -6$ and $c_2 = 0$.

\paragraph{Explicit function of $p$} There is a clear compromise between the two strategies above. We propose instead a middle ground, where $c_1$ and $c_2$ are ``simple,'' explicit functions of $p$. The decision variables depend on the data, but they are explicit, low-complexity computable functions.

What would be a good definition for $c_1$ and $c_2$ of this form? A ``good choice'' for $c_1$ and $c_2$ is one with the following informally stated properties:\begin{itemize} \item Non-inferiority: it should be no worse than fixing $c_1 = c_2 = 0$ and \item Maximality: it should be maximal in some sense.\end{itemize} Formal statements of these properties and additional intuition for our choices are presented in Section~\ref{sec:why}.

For the cubic case, we propose the choice
\begin{equation}
c_1 = -2\sqrt{3p_0(p_2)_+} 
\quad \text{   and   } \quad 
c_2 = -2\sqrt{3p_3(p_1)_+}.
\label{eqn:cubic_c1c2}
\end{equation} 
While this is more complicated than $c_1 = c_2 = 0$, Section~\ref{sec:why} provides additional motivation and an argument of why our proposal satisfies the desirata mentioned above. Substituting the values in \eqref{eqn:cubic_c1c2} into \eqref{eqn:nonnegative_socp} gives the conditions that define our new set.
\begin{defn}
    The set $\mathcal{GB}$ (``Geometric Bernstein'') is
    \begin{equation}\label{eqn:definingGB}
        \mathcal{GB} = \Bigg\{ \sum_{i=0}^3p_ib_i^3(x) \quad \Bigg| \quad p_0, p_3 \geq 0, 
        \qquad p_2 \geq -2\sqrt{\frac{p_3(p_1)_+}{3}},
        \qquad p_1 \geq -2\sqrt{\frac{p_0(p_2)_+}{3}}  \Bigg\}.
    \end{equation}
\end{defn}%
\noindent Each Bernstein coefficient is bounded below by a multiple of the geometric mean of its neighboring Bernstein coefficients, which motivates the name of the set. A visual comparison of~$\mathcal{GB}$ to~$\mathcal{NB}$ and~$\mathcal{P}$ is given in Figure~\ref{fig:inclusion-cross-sections11}.
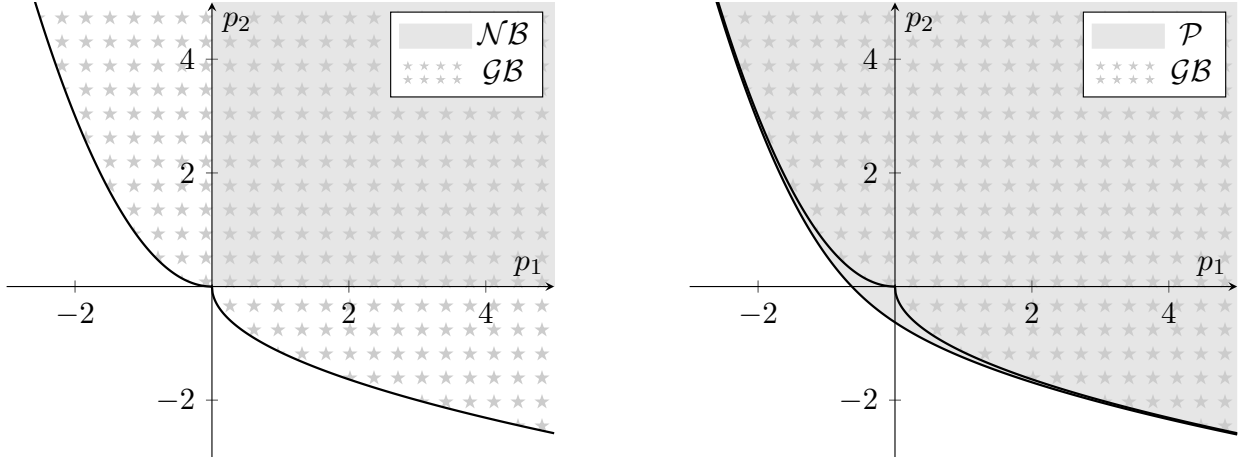
\begin{figure}
    \centering

    \begin{subfigure}{.45\textwidth} 
    \resizebox{\textwidth}{!}{\begin{tikzpicture}[
    declare function={
    func(\x)= (\x < 0) *(3/4*\x^2)   + (\x >= 0)*(-2*sqrt(\x/3))
   ;}]
    \begin{axis}[axis on top,
            xmin=-3, xmax=5,
            ymin=-3, ymax=5,
            xlabel=$p_1$,
            ylabel=$p_2$,
            axis lines=middle,
            clip = true, 
            xticklabels={},
            yticklabels={},
            extra x ticks={-2,2,4},
            extra y ticks={-2,2,4},
            legend image post style={scale=1.5},
        star legend/.style={
          legend image code/.code={
          \foreach \x in {0.01,0.03,.05,.07}
            {\node[star, star points=5, star point ratio=3, draw=gray!40, fill=gray!40, minimum size = 1mm, anchor=center, scale = 0.1] at (\x,-.01) {};}
            \foreach \x in {0.01,0.03,.05,.07}
            {\node[star, star points=5, star point ratio=3, draw=gray!40, fill=gray!40, minimum size = 1mm, anchor=center, scale = 0.1] at (\x,.01) {};}
          }
        }        ]
        \addplot [thick, samples = 1001, smooth, domain = -3:5, name path global = GB, forget plot] {func(x)};
  
        \path[name path global=tophalf] (axis cs:-3,5.1) -- (axis cs:5,5.1);

        \path[name path global=xaxis]
        (axis cs:-3,0) -- (axis cs:5, 0);
        \addplot[gray!20] fill between[of=tophalf and xaxis,
        soft clip = {domain=0:5},
        ];
            
        \addplot[forget plot] fill between[of=tophalf and GB,
        split,
        every segment no 1/.style={pattern=fivepointed stars,pattern color=gray!40},
        ];  
        \addlegendimage{star legend}
        \legend{
            $\mathcal{NB}$,
            $\mathcal{GB}$
        };
    \end{axis}
\end{tikzpicture}
}

    \end{subfigure}
\hfill
    \begin{subfigure}{.45\textwidth}\resizebox{\textwidth}{!}{\begin{tikzpicture}[
    declare function={
    func(\x)= (\x < 0) *(3/4*\x^2)   + (\x >= 0)*(-2*sqrt(\x/3));
   }
   ]
    \begin{axis}[axis on top,
            xmin=-3, xmax=5,
            ymin=-3, ymax=5,
            xlabel=$p_1$,
            ylabel=$p_2$,
            axis lines=middle,
            clip = true, 
            xticklabels={},
            yticklabels={},
            extra x ticks={-2,2,4},
            extra y ticks={-2,2,4},
            legend image post style={scale=1.5},
        star legend/.style={
          legend image code/.code={
          \foreach \x in {0.01,0.03,.05,.07}
            {\node[star, star points=5, star point ratio=3, draw=gray!40, fill=gray!40, minimum size = 1mm, anchor=center, scale = 0.1] at (\x,-.01) {};}
            \foreach \x in {0.01,0.03,.05,.07}
            {\node[star, star points=5, star point ratio=3, draw=gray!40, fill=gray!40, minimum size = 1mm, anchor=center, scale = 0.1] at (\x,.01) {};}
          }
        }        ]

    \addplot[name path global = POS, black, thick, smooth, forget plot] table[x=x,y=y, col sep=comma] {tikz-plots/disc-values.csv};
    
    \addplot [thick, samples = 1001, smooth, domain = -3:5, name path global = GB, forget plot] {func(x)};
  
        \path[name path global=tophalf] (axis cs:-3,5.1) -- (axis cs:5,5.1);

        \path[name path global=xaxis]
        (axis cs:-3,0) -- (axis cs:5, 0);
        \addplot[gray!20] fill between[of=POS and tophalf
        ];
        \addplot[forget plot] fill between[of=tophalf and GB,
        split,
        every segment no 1/.style={pattern=fivepointed stars,pattern color=gray!40},
        ];  
        \addlegendimage{star legend}
        \legend{
            $\mathcal{P}$,
            $\mathcal{GB}$
        };
    \end{axis}
\end{tikzpicture}
}
    \end{subfigure}
    \caption{Membership in $\mathcal{NB}$, $\mathcal{GB}$, and $\mathcal{P}$ for a two dimensional slice ($p_0 = p_3 = 1$) of all cubic polynomials. The axes of the graphs are $p_1$ and $p_2$. A point is shaded depending on whether $b_0(x) + p_1b_1(x) + p_2b_2(x) + b_3(x)$ is in $\mathcal{NB}$, $\mathcal{GB}$, or $\mathcal{P}$. The inclusions $\mathcal{NB} \subset \mathcal{GB} \subset\mathcal{P}$ are evident, and for this slice each inclusion is strict. We see that the set $\mathcal{GB}$ is not convex.
    }\label{fig:inclusion-cross-sections11}
\end{figure}

\subsection{Why this definition?} \label{sec:why}

\paragraph{Non-inferiority} The following theorem formalizes the claim that our choices of $c_1$ and $c_2$ are at least as good as $c_1 = c_2 = 0$.

\begin{restatable}{thm}{thmoneintwointhree}\label{thm:1in2in3}
    $\mathcal{NB} \subset \mathcal{GB}\subset \mathcal{P}$
\end{restatable}
\begin{proof}[Proof of Theorem~\ref{thm:1in2in3}]
    The containment $\mathcal{NB} \subset \mathcal{GB}$ is simple: If all $p_i \geq 0$, the left hand side of all inequalities in \eqref{eqn:definingGB} are nonnegative and the right hand sides are nonpositive. Therefore $p \in \mathcal{GB}$.
    
    Next, we show that $\mathcal{GB} \subset \mathcal{P}$.
    Let $p \in \mathcal{GB}$. We claim setting $c_1$ and $c_2$ as in \eqref{eqn:cubic_c1c2} satisfies \eqref{eqn:nonnegative_socp}. The constraint $(p_0, p_2 - \frac{c_2}{3}; \frac{c_1}{\sqrt{6}}) \in \mathcal{Q}$ is a combination of three inequalities:
    \begin{itemize}
        \item $p_0 \geq 0$: This is required by membership in $\mathcal{GB}$.
        \item $p_2 - \frac{c_2}{3} \geq 0$: Since $p \in \mathcal{GB}$, $0 \leq p_2 + 2\sqrt{\frac{(p_1)_+p_3}{3}} = p_2 - \frac{c_2}{3}$.
        \item $2p_0(p_2 - \frac{c_2}{3}) \geq \frac{c_1^2}{6}$: First consider the case $p_2 \geq 0$. We always have  $c_2 \leq 0$, so $2p_0(p_2 - \frac{c_2}{3}) \geq 2p_0p_2 = \frac{c_1^2}{6}$. In the other case $p_2 < 0$ and so $c_1 = 0$. In that case the first two inequalities prove the product is nonnegative. 
    \end{itemize}
    Therefore this second order cone condition is satisfied. The argument that the other second order cone condition holds is analogous.\end{proof}

A less precise, but more insightful, ``picture proof'' that $\mathcal{NB} \subset \mathcal{GB}$ is given in Figure~\ref{fig:abchoice}, which also motivates the choice of $c_1$ and $c_2$. The feasible region of each second order cone condition is the shaded region inside of each parabola. We want a point $(c_1,c_2)$ inside the intersection of these two regions. The set $\mathcal{NB}$ results from picking the origin. The $\star$ will always be in the intersection whenever the origin is.

The intuition for why $\mathcal{GB}$ strictly contains $\mathcal{NB}$ is the following. If $p_1 < 0$, the horizontal parabola moves left. Then we would choose $c_2 = 0$, but picking a negative $c_1$ may still be in the intersection even if the origin is not.

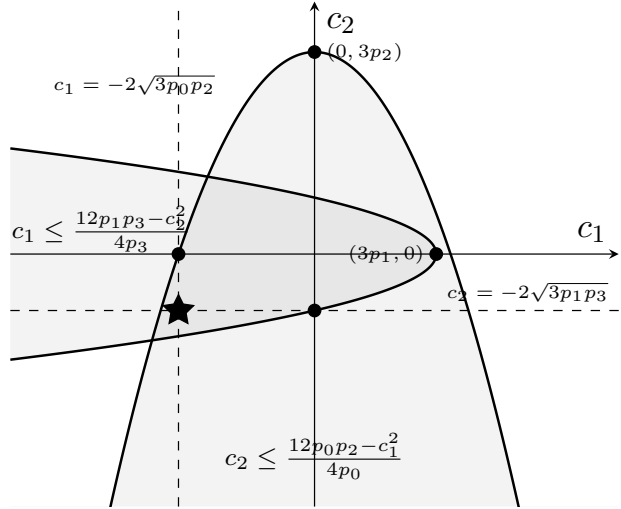
\begin{figure}
    \centering
    \resizebox{.5\textwidth}{!}{\begin{tikzpicture}
    \begin{axis}[
            xmin=-1, xmax=1,
            ymin=-1, ymax=1,
            xlabel=$c_1$,
            ylabel=$c_2$,
            axis lines=middle,
            clip=true,
            xticklabels={},
            yticklabels={},
            xtick = \empty,
            ytick = \empty
        ]
        \addplot[
            name path global=vert,
            domain=-1:1,
            samples=100,
            color=black,
            thick,
            smooth
        ]
        {.8-4*x^2};
        \path[name path global=bottom] (axis cs:-1,-1) -- (axis cs:1,-1);
        \addplot [
            thick,
            color=black,
            fill=black, 
            fill opacity=0.05
        ]
        fill between[
            of=vert and bottom,
        ];
        \addplot[
            name path global=tophalf,
            domain=-1:.4,
            samples=1000,
            color=black,
            thick,
            smooth
        ]
        {sqrt(-(x-.4)/8};
        \addplot[
            name path global=bottomhalf,
            domain=-1:.4,
            samples=1000,
            color=black,
            thick,
            smooth
        ]
        {-sqrt(-(x-.4)/8};
        \addplot [
            thick,
            color=black,
            fill=black, 
            fill opacity=0.05
        ]
        fill between[
            of=tophalf and bottomhalf,
        ];
        \node at (0, -.8)    {\scriptsize{$c_2 \leq \frac{12p_0p_2 - c_1^2}{4p_0}$}};
        \node at (-.7, .1)    {\scriptsize{$c_1 \leq \frac{12p_1p_3 - c_2^2}{4p_3}$}};
        \filldraw[black] (0,.8) circle (2pt) node[anchor=west]{\tiny{$(0,3p_2)$}};
        \filldraw[black] (.4,0) circle (2pt) node[anchor=east]{\tiny{$(3p_1,0)$}};
        \draw[dashed] (-0.4472135955,-1)--(-0.4472135955,1);
        \filldraw[black] (-0.4472135955,0) circle (2pt);
        \node[anchor=20, circle] at (-.3,.8) {\tiny{$c_1 = -2\sqrt{3p_0p_2}$}};
        \draw[dashed] (-1,-0.22360679775)--(1,-0.22360679775);
        \filldraw[black] (0, -0.22360679775) circle (2pt);
        \node at (.7,-.15) {\tiny{$c_2 = -2\sqrt{3p_1p_3}$}};
        \node[draw, star, solid, black, scale = .5, star point ratio = 2, star points = 5, fill = black] at (-0.4472135955,-0.22360679775) {};
    \end{axis}
\end{tikzpicture}
}
    \caption{The feasibility region of $(c_1,c_2)$ for Lemma~\ref{lem:nonnegative_socp} when $p_i > 0$. Re-arrange the conditions given by the second order cone problem to get $c_2 \leq \frac{12p_0p_2 - c_1^2}{4p_0}$ and $c_1 \leq \frac{12p_1p_3 - c_2^2}{4p_3}$. The region of feasibility for $c_1$ and $c_2$ is the intersection of the shaded regions bounded by the parabolas. When $c_1= -2\sqrt{3p_0p_2}$, then as long as $c_2 \leq 0$, the point will be in the feasible region of the vertical parabola. Similar reasoning follows for the horizontal parabola. The choice made for $\mathcal{GB}$ is marked with a star, as opposed to the choice $(0,0)$ used for $\mathcal{NB}$.}
    \label{fig:abchoice}
\end{figure}

\paragraph{Maximality} Now that we have established our choice of $c_1$ and $c_2$ is ``better'' than $c_1 =c_2 = 0$, we formalize the sense in which our choice is maximal. In general, $c_1$ and $c_2$ could be functions of all of the Bernstein coefficients; i.e., $c_1 = g(p_0,p_1,p_2,p_3)$ and $c_2 = h(p_0,p_1,p_2,p_3)$.\footnote{For instance, if $g$ and $h$ were defined as giving the analytic center  \cite[\S8.5]{BoydBook} of the convex set \eqref{eqn:nonnegative_socp}, then that $g$ and $h$ would always give $c_1$ and $c_2$ that prove nonnegativity whenever $p$ is nonnegative on the interval. This yields $g$ and $h$ that are semialgebraic functions, but they would be ``too complicated'' for a simple test.} For the sake of simplicity our choice of $g$ and $h$ will be such that they are only functions of \emph{two} of the Bernstein coefficients each. Furthermore, by symmetry, we restrict our attention to those where $c_1 = g(p_0,p_1,p_2,p_3) = f(p_0,p_2)$ and $c_2 = h(p_0,p_1,p_2,p_3) = f(p_3,p_1)$ for some $f :\mathbb{R}^2 \to \mathbb{R}$. Let $S(f)$ be the set
\begin{equation}
    S(f) = \Big\{ p \in \mathbb{R}_3[x] \, \Big| \, c_1 = f(p_0,p_2) \text{ and } c_2 = f(p_1, p_3) \text{ are feasible for \eqref{eqn:nonnegative_socp}} \Big\}.
\end{equation}

Using a single $f$ for both $c_1$ and $c_2$ ensures that membership in $S(f)$ does not change when transforming $p$ by $x \to 1-x$. For our choice in \eqref{eqn:cubic_c1c2}, we take $f(z_1,z_2) = -2\sqrt{3(z_1)_+(z_2)_+}$, which we denote as $f^*$. For reference, $S(f^*) = \mathcal{GB}$ and $S(0) = \mathcal{NB}$. Within this family, $S(f^*)$ is maximal in the sense of the following proposition, which is proven in Appendix~\ref{sec:p2proofs}.

\begin{restatable}{prop}{tight}\label{prop:tight} 
If there exists a $p \in \mathcal{GB}\setminus \mathcal{NB}$ such that both $f(p_0, p_2) \neq f^*(p_0, p_2)$ and $f(p_3, p_1) \neq f^*(p_3, p_1)$, then
$S(f^*) \not\subset S(f)$. Equivalently, if $S(f^*) \subset S(f)$ then for every $p \in \mathcal{GB} \setminus \mathcal{NB}$, either $f(p_0,p_2) = f^*(p_0,p_2)$ or $f(p_3, p_1) = f^*(p_3,p_1)$.
\end{restatable}

\subsection{Aside: an exact condition with the discriminant}\label{sec:cubic-exact}

One can also characterize strict positivity through the \emph{discriminant} of a polynomial. As opposed to the second order cone constraints that require solving a convex optimization problem, this condition can be checked explicitly. The discriminant is defined as follows.
\begin{defn}
    The \emph{discriminant} of a polynomial $p$ with leading coefficient $a_n$ in the monomial basis and roots $\alpha_i$ for $1\leq i \leq n$ is
    \begin{equation}\label{eqn:discdef}
        D(p) = a_n^{2n-2}\prod_{1 \leq i < j \leq n}(\alpha_i - \alpha_j)^2.
    \end{equation}
\end{defn}
\noindent Although not obvious from this definition, $D(p)$ is a polynomial in the (monomial or Bernstein) coefficients of $p$.

\noindent For now we are only concerned with cubic polynomials, so we explicitly compute the discriminant for this case in terms of Bernstein coefficients.
\begin{example}
The discriminant of the polynomial $\sum_{i=0}^3 p_i b_i(x)$ is
\begin{equation}
    -27 (-3 p_1^2 p_2^2 + 4 p_0p_2^3 + 4 p_1^3 p_3 - 6 p_0p_1p_2p_3 + p_0^2 p_3^2).
\end{equation}
\end{example}

\begin{figure}[tbp]
    \centering
    \begin{subfigure}[t]{.45\textwidth}
    \begin{tikzpicture}[
    declare function={
    func(\x)= (\x < 0) *(3/4*\x^2)   + (\x >= 0)*(-2*sqrt(\x/3));
   }
   ]
    \begin{axis}[axis on top,
    legend pos = south west,
            xmin=-3, xmax=5,
            ymin=-3, ymax=5,
            xlabel=$p_1$,
            ylabel=$p_2$,
            axis lines=middle,
            clip = true, 
            xticklabels={},
            yticklabels={},
            extra x ticks={-2,2,4},
            extra y ticks={-2,2,4},
            legend image post style={scale=1.5},
        star legend/.style={
          legend image code/.code={
          \foreach \x in {0.01,0.03,.05,.07}
            {\node[star, star points=5, star point ratio=3, draw=gray!40, fill=gray!40, minimum size = 1mm, anchor=center, scale = 0.1] at (\x,-.01) {};}
            \foreach \x in {0.01,0.03,.05,.07}
            {\node[star, star points=5, star point ratio=3, draw=gray!40, fill=gray!40, minimum size = 1mm, anchor=center, scale = 0.1] at (\x,.01) {};}
          }
        }        ]

    \addplot[name path global = BOTTOM, black, thick, smooth, forget plot] table[x=x,y=y, col sep=comma] {tikz-plots/disc-values.csv};
    \addplot[name path global = BOTTOMRIGHT, black, thick, smooth, forget plot] table[x=x,y=y, col sep=comma] {tikz-plots/disc-values2.csv};
    \addplot[name path global = TOPRIGHT, black, thick, smooth, forget plot] table[x=x,y=y, col sep=comma] {tikz-plots/disc-values3.csv};
    
        \path[name path global=tophalf] (axis cs:-2.60384,5.1) -- (axis cs:5,5.1);

        \path[name path global=xaxis]
        (axis cs:-3,0) -- (axis cs:5, 0);
        \addplot[gray!20] fill between[of=tophalf and BOTTOM
        ]; \addlegendentry{$\mathcal{P}$}
        \addplot[pattern={Lines[angle=45,distance=8pt]}] fill between[of=tophalf and TOPRIGHT
        ]; \addlegendentry{$\mathcal{D}$}

        \addplot[pattern={Lines[angle=45,distance=8pt]}] fill between[of=BOTTOM and BOTTOMRIGHT
        ];
    \end{axis}
\end{tikzpicture}

    \caption{The two-dimensional slice $p_0 = p_3 = 1$. A point is shaded if $b_0(x) + p_1b_1(x) + p_2b_2(x)+b_3(x)$ is a polynomial in $\mathcal{P}$ or $\mathcal{D}^\geq$.} \label{fig:NB_D}
    \end{subfigure}
    \hfill
    \begin{subfigure}[t]{.45\textwidth}
        \begin{tikzpicture}[
declare function={
func(\x)= (\x < 0) *(3/4*\x^2)   + (\x >= 0)*(-2*sqrt(\x/3));
}
]
\begin{axis}[axis on top,
legend pos = south west,
        xmin=-3, xmax=5,
        ymin=-3, ymax=5,
        xlabel=$p_1$,
        ylabel=$p_2$,
        axis lines=middle,
        clip = true, 
        xticklabels={},
        yticklabels={},
        extra x ticks={-2,2,4},
        extra y ticks={-2,2,4},
        legend image post style={scale=1.5},
    star legend/.style={
      legend image code/.code={
      \foreach \x in {0.01,0.03,.05,.07}
        {\node[star, star points=5, star point ratio=3, draw=gray!40, fill=gray!40, minimum size = 1mm, anchor=center, scale = 0.1] at (\x,-.01) {};}
        \foreach \x in {0.01,0.03,.05,.07}
        {\node[star, star points=5, star point ratio=3, draw=gray!40, fill=gray!40, minimum size = 1mm, anchor=center, scale = 0.1] at (\x,.01) {};}
      }
    }        ]

    \addplot[name path global = parabola, domain = -3:3, black, thick, forget plot] {3*x^2/4};
    \addplot[domain = -3:5, black, thick, forget plot] {0};

    \path[name path global=tophalf, black, thick] (axis cs:0,5.1) -- (axis cs:5,5.1);

    \path[name path global=xaxis]
    (axis cs:0,0) -- (axis cs:5, 0);
    \addplot[gray!20] fill between[of=tophalf and parabola
    ]; \addlegendentry{$\mathcal{P}$}
    \addplot[gray!20, forget plot] fill between[of=tophalf and xaxis
    ];
    \addplot[pattern={Lines[angle=45,distance=8pt]}] fill between[of=tophalf and parabola
    ]; \addlegendentry{$\mathcal{D}$}

    \node at (-2,0)[label=$r$, circle,fill,inner sep=1.5pt]{};
    \node at (-2,3)[label=left:$s$, circle,fill,inner sep=1.5pt]{};


\end{axis}
\end{tikzpicture}

        \caption{The two-dimensional slice $p_0 = 1,p_3 = 0$. A point is shaded if $b_0(x) + p_1b_1(x) + p_2b_2(x)$ is a polynomial in $\mathcal{P}$ or $\mathcal{D}^\geq$.} 
        \label{fig:NB_D2}
        \end{subfigure}
    \caption{The relationship between $\mathcal{P}$ and the discriminant for two slices of the set of cubic polynomials. The axes of the graphs are $p_1$ and $p_2$. The bold black curves mark where the discriminant vanishes. Figure~\ref{fig:NB_D} shows that $\mathcal{D}^\geq$ does not contain all of $\mathcal{P}$. We remark that although the discriminant is irreducible, it may factor when restricted to a two-dimensional slice, as in Figure~\ref{fig:NB_D2}.  This also shows that the discriminant can vanish even for polynomials that are not on the boundary of $\mathcal{P}$. The dots mark the polynomials $r$ and $s$ used in Examples~\ref{example:strictinterior1} and \ref{example:strictinterior2}.
    }
    \label{fig:NB&Dboth}
\end{figure}
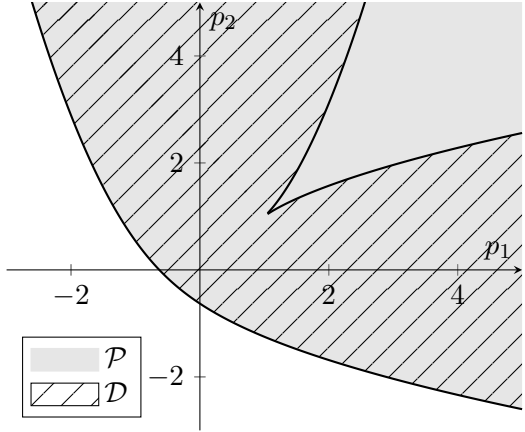
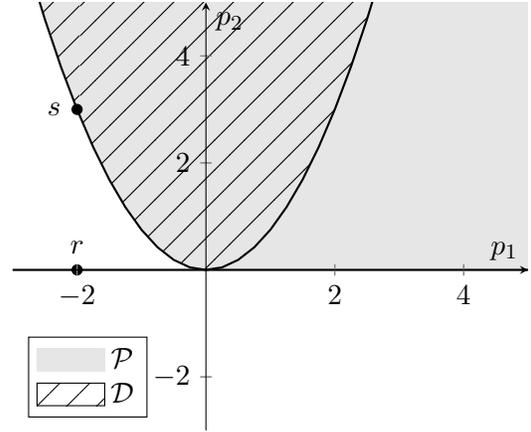

Polynomials with repeated roots provide an intuitive connection between the discriminant and the set of nonnegative polynomials. On the one hand, the discriminant of $p$ vanishes precisely when~$p$ has repeated roots, as can be seen directly from the definition in \eqref{eqn:discdef}. On the other hand, the boundary of the set of polynomials nonnegative on $\mathbb{R}$ consists of polynomials with repeated roots. This connection via polynomials with repeated roots suggests that changes in nonnegativity can be characterized by changes in sign of the discriminant.

There is a caveat for using this intuition for nonnegativity on the interval: it matters whether the repeated root is actually \emph{on} the interval -- it could be outside, or it could even be complex-valued. In those cases the discriminant would vanish, but $p$ may not on the boundary of $\mathcal{P}$. This concern is addressed by ensuring additional conditions, not just sign conditions on $D(p)$.

Theorem~\ref{thm:exactcubic} carefully formalizes the connection but requires the definition of a few sets. The interior of the set $\mathcal{P}$, which we denote $\mathcal{P}^\circ$, is the set of strictly positive cubic polynomials on the closed interval $[0,1]$. Similarly, $\mathcal{NB}^\circ$ is the set of cubic polynomials with strictly positive Bernstein coefficients. Let $\mathcal{D}^{>}$ be the set of cubic polynomials $p$ such that $p_0, p_3 > 0$ and $-D(p) > 0$. Let $\mathcal{D}^{\geq}$ be the set of cubic $p$ polynomials such that $p_0,p_3 \geq0$ and $-D(p) \geq 0$.\footnote{In fact, $\mathcal{D}^\geq$ is not the closure of $\mathcal{D}^>$. To see this, note the polynomial $r$ in Example~\ref{example:strictinterior1} is in $\mathcal{D}^\geq$ but is not the limit point of a sequence in $\mathcal{D}^>$.} It is evident from Figure~\ref{fig:NB_D} that $\mathcal{D}^>$ is not all of $\mathcal{P}^\circ$ (because $\mathcal{P}^\circ$ includes $\mathcal{NB}^\circ$). All that is needed is to include $\mathcal{NB}^\circ$: 
 
\begin{restatable}{thm}{exactcubic}\label{thm:exactcubic}
    $\mathcal{P}^\circ = \mathcal{D}^> \cup \mathcal{NB}^\circ$.
\end{restatable}
On the other hand, $\mathcal{P} \neq \mathcal{D}^\geq \cup \mathcal{NB}$.\footnote{It is easy to overlook these subtle complications on the boundary, but doing so could lead to imprecise results such as Proposition 2 of \cite{schmidt1988positivity}.} 
If the discriminant of $p$ vanishes when $p \not \in \mathcal{NB}$, it is not possible to characterize the sign of $p$ solely based on this information. We can see this in Figure~\ref{fig:NB_D2}. The negative horizontal axis has vanishing discriminant but does not contain nonnegative polynomials. The left hand side of the parabola consists of nonnegative polynomials. To be completely explicit, we provide two polynomials, one from each of these sets, marked with a dot. These examples show that both $\mathcal{P}$ and its complement have a nonempty intersection with $\mathcal{D}^\geq$.  
\begin{example}\label{example:strictinterior1}
    Let $r = (1-x)^3 - 6x(1-x)^2$. Not all the Bernstein coefficients are positive and the discriminant of $r$ is zero, but $r(\frac{1}{2}) = -\frac{5}{8}$, so $r$ is not positive. 
\end{example}

\begin{example}\label{example:strictinterior2}
    Let $s = (1-x)^3-6x(1-x)^2+3x^2(1-x) = (1-x)(1-4x)^2$. Not all the Bernstein coefficients are positive and the discriminant of $s$ is zero. Nevertheless, the factored form demonstrates nonnegativity of $s$ on the interval.
\end{example}
\noindent While the exact condition discussed in this section has an explicit form, it only characterizes strictly positive polynomials.

A traditional proof of Theorem~\ref{thm:exactcubic} would involve a somewhat tedious enumeration of cases. Instead of that route, we take a more modern computational approach and use a quantifier elimination program to ``derive'' the theorem. See Appendix~\ref{sec:cubic-exact-proof} for a demonstration. 

\subsection{An exact condition with \texorpdfstring{$\mathcal{GB}$}{GB}}

One can also use $\mathcal{GB}$ to provide an exact characterization of nonnegativity. Recall that we have the strict containment $\mathcal{GB} \subset \mathcal{P}$. Perhaps surprisingly, any element in $\mathcal{P}$ is actually the sum of two elements in the smaller set $\mathcal{GB}$, i.e.,
\begin{equation}
\label{eqn:gbplusgb}
\mathcal{GB} + \mathcal{GB} = \mathcal{P}.
\end{equation}
The following proposition shows how to construct $\mathcal{P}$ from two convex subsets of $\mathcal{GB}$ and  yields~\eqref{eqn:gbplusgb} as a corollary.
\begin{prop}
    Let \begin{equation}\begin{aligned}
        S = &\Bigg\{\sum_{i=0}^3s_ib_i^3(x) \, \Bigg | \,  \Big(s_0,\, s_2;\, \sqrt{\frac32}s_1\Big) \in \mathcal{Q}  \textrm{ and } s_3 \geq 0\Bigg\} \textrm{ and }\\
        T = &\Bigg\{\sum_{i=0}^3t_ib_i^3(x) \, \Bigg | \,  \Big(t_1,\, t_3;\, \sqrt{\frac32}t_2\Big) \in \mathcal{Q}  \textrm{ and } t_0 \geq 0\Bigg\}.
        \end{aligned}\end{equation}
        Then, $S + T = \mathcal{P},$ which implies $\mathcal{GB} + \mathcal{GB} = \mathcal{P}$.
\end{prop}

\begin{proof}
    Since $S, T \subseteq \mathcal{GB}$, and $\mathcal{GB} \subset \mathcal{P}$, we have $S, T \subset \mathcal{P}$. Furthermore, $\mathcal{P}$ is a cone, so $S + T \subset \mathcal{P}$. 
    For the other direction, suppose that $p \in \mathcal{P}$. We will decompose $p = s + t$ and show $s \in S$ and $t \in T$. Let $c_1, c_2$ be the constants guaranteed to exist for $p$ by Lemma~\ref{lem:nonnegative_socp}. Let the Bernstein coefficients of $s$ be $(p_0, \frac{c_1}{3}, p_2 - \frac{c_2}{3}, 0)$ and let the Bernstein coefficients of $t$ be $(0, p_1 - \frac{c_1}{3}, \frac{c_2}{3}, p_3)$. All that is left is to show that $s \in S$ and $t \in T$. First, $s \in S$ because $(s_0, s_2; \sqrt{\frac{3}{2}}s_1) \in \mathcal{Q}$ if and only if $(p_0, p_2 - \frac{c_2}{3}; \frac{c_1}{\sqrt{6}}) \in \mathcal{Q}$, which is guaranteed by the choice of $c_1$ and $c_2$. The same argument shows $t \in T$. Hence the decomposition is valid. Therefore, $P \subset S + T$.
\end{proof}

\section{Arbitrary degree}\label{sec:generalization}

Now we extend the definitions of $\mathcal{NB}$, $\mathcal{P}$, and eventually $\mathcal{GB}$ to arbitrary degree.
\begin{defn} $\mathcal{P}_d$ (``Positive'') is the set of $p \in \mathbb{R}_d[x]$ such that $p$ is nonnegative on $[0,1]$.
\end{defn}
\begin{defn} $\mathcal{NB}_d$ (``Nonnegative Bernstein'') is the set of $p \in \mathbb{R}_d[x]$ such that $p$ has all nonnegative Bernstein coefficients.
\end{defn} 
In this section, we propose the corresponding generalization of $\mathcal{GB}$. To accomplish this, one can attempt to generalize Lemma~\ref{lem:nonnegative_socp} for arbitrary degree. However, whereas in the cubic case the second order cone conditions are necessary and sufficient for nonnegativity, in the general case this requires instead a semidefinite program. This is too computationally expensive, so we restrict to only \emph{tridiagonal} matrices, which will give an SOCP condition. We then make strategic choices for the decision variables in the SOCP. 

To ease notation, define \begin{equation}\label{eqn:widef}w_i = \begin{cases}\frac{1}{2} &\qquad  1 < i < d-1  \\ 1 &\qquad \text{otherwise}\end{cases}.\end{equation} Recall that $m_i = \frac12 \cdot \frac{(i+1) \cdot (d-i+1)}{i \cdot (d-i)}$. Here is a sufficient condition for nonnegativity written in terms of second order cones and Bernstein coefficients.
\begin{restatable}{lem}{gennonnegativesocp}\label{lem:gen_nonnegative_socp} Let $p\in \mathbb{R}_d[x]$ with Bernstein coefficients $\{p_i\}_{i=0}^{d}$. Let $c_0 = c_d = 0$. If there exist $c_1, \dots, c_{d-1} \in \mathbb{R}$ such that for all $1 \leq i \leq d-1$
\begin{equation}\label{eqn:general-scalar-sufficient}
\begin{aligned}
\left( p_{i-1}-c_{i-1}, \quad p_{i+1} - c_{i+1};\quad c_i\sqrt{\frac{m_i}{w_{i-1}w_{i+1}}} \right) \quad  &\in \quad \mathcal{Q},
\end{aligned}
\end{equation}
then $p\in \mathcal{P}_d$.
\end{restatable}
\begin{proof}
Throughout this proof we use $b_i$ to mean $b_i(x)$ for any $0 \leq x \leq 1$. Constraint \eqref{eqn:general-scalar-sufficient} is equivalent to 
\begin{equation}\label{eqn:wgeneral-scalar-sufficient}
    \left( w_{i-1}(p_{i-1}-c_{i-1}), \quad w_{i+1}(p_{i+1} - c_{i+1});\quad c_i\sqrt{m_i} \right) \quad \in \quad \mathcal{Q}.
\end{equation}
By Lemma~\ref{lemma:bernstein_socp}, for all $1 \leq i \leq d-1$, $(b_{i-1}, b_{i+1}; \frac{1}{\sqrt{m_i}}b_i) \in \mathcal{Q}$.

The inner product between constraint $i$ of Lemma~\ref{lemma:bernstein_socp} and the $i$th constraint of \eqref{eqn:wgeneral-scalar-sufficient} is
\begin{equation}
    w_{i-1}(p_{i-1} - c_{i-1})b_{i-1} + w_{i+1}(p_{i+1} - c_{i+1})b_{i+1} + c_ib_i \geq 0
\end{equation}
by self-duality of the second order cone. Sum all of these results from $1 \leq i \leq d-1$ to get
\begin{equation}
\begin{aligned}\label{eqn:lemma4-allbutspecial}
    0 &\leq \sum_{i=1}^{d-1}\left(w_{i-1}(p_{i-1} - c_{i-1})b_{i-1} + w_{i+1}(p_{i+1} - c_{i+1})b_{i+1} + c_ib_i \right) \\
    &= \sum_{i=0}^kp_ib_i.%
    \end{aligned}%
\end{equation}%
\end{proof}
\noindent 
In Appendix~\ref{sec:altsosproof} we give an alternative proof of the lemma that shows \eqref{eqn:general-scalar-sufficient} is a restricted case of sums of squares, where the Gram matrix is constrained to be tridiagonal.

Unlike in the cubic case, Lemma~\ref{lem:gen_nonnegative_socp} is only a sufficient condition for nonnegativity. To emphasize that it is sufficient but not necessary, consider the following example of a nonnegative polynomial for which \eqref{eqn:general-scalar-sufficient} is not feasible. 

\begin{example} Let $p = (2x-1)^4$. Clearly $p(x) \geq 0$. The Bernstein coefficients are $(1, -1,1, -1, 1)$,
but there are no feasible $c_i$ for \eqref{eqn:general-scalar-sufficient}.
\end{example}
Just as in the cubic case, we can produce sufficient conditions for nonnegativity as follows. Fix an explicit formula for $c_i$ in Lemma~\ref{lem:gen_nonnegative_socp}. If a polynomial satisfies those fixed constraints then $p \in \mathcal{P}_d$. Setting $c_i = 0$ in this procedure gives the defining conditions of $\mathcal{NB}_d$. We propose the choice \begin{equation}\label{eqn:arbitrary_ci}c_i = -\sqrt{\frac{2w_{i-1}w_{i+1}}{m_i}(p_{i-1})_+(p_{i+1})_+},\end{equation} which gives rise to the condition below. Define $p_{i} = 0$ if $i<0$ or $i > d$. 
\begin{defn}
    $\mathcal{GB}_d$ is the set of $p \in \mathbb{R}_d[x]$ for which for every $0 \leq i \leq d$, \begin{equation}\label{eqn:scalar-ci}p_i \geq -\sqrt{\frac{2w_{i-1}w_{i+1}}{m_i}(p_{i-1})_+(p_{i+1})_+} .\end{equation} 
\end{defn}

The following theorem shows that this choice of $c_i$ is not worse than $c_i = 0$. 
\begin{restatable}[Generalized Theorem~\ref{thm:1in2in3}]{thm}{genoneintwointhree} \label{thm:1in2in3_general}
    $\mathcal{NB}_d \subset \mathcal{GB}_d \subset \mathcal{P}_d$.
\end{restatable}
\begin{proof}
For this proof we drop the subscript $d$ but it is understood that all polynomials are degree $d$. It is clear that if every $p_i \geq 0$, then they are all at least some nonpositive number, so $p\in \mathcal{NB}$.

Next, we show that $\mathcal{GB} \subset \mathcal{P}$. Let $p \in \mathcal{GB}_d$. We claim that setting $c_i$ as in \eqref{eqn:arbitrary_ci} for every $i$ makes \eqref{eqn:general-scalar-sufficient} feasible. The second order cone condition is satisfied if and only if both of the following hold:
\begin{enumerate}
    \item $p_{i-1} \geq c_{i-1}$ and $p_{i+1} \geq c_{i-1}$: These are implied by the definition of $p \in \mathcal{GB}$.
    \item $2w_{i-1}w_{i+1}(p_{i-1}-c_{i-1})(p_{i+1}-c_{i+1}) \geq m_i c_i^2$: If either $p_{i-1}$ or $p_{i+1}$ are not positive, then $c_i = 0$. Since the two factors are nonnegative, the inequality holds in that case. Otherwise, when $p_{i-1}, p_{i+1} \geq 0$, we have 
    \begin{equation*}
        2w_{i-1}w_{i+1}(p_{i-1}-c_{i-1})(p_{i+1}-c_{i+1}) \geq 2w_{i-1}w_{i+1}(p_{i-1}p_{i+1}) = m_ic_i^2
        \end{equation*}
        due to the nonpositivity of $c_{i-1}$ and $c_{i+1}$ and the definition of $c_i^2$. 
\end{enumerate}
Therefore, the second order cone condition is feasible with this choice, and so $p \in \mathcal{P}$.
\end{proof}
\section{Polynomial matrices}\label{sec:polynomial-matrices}

The same problems and approaches we considered for scalar polynomials generalize to matrix-valued polynomials, or polynomial matrices.
A (symmetric) polynomial matrix $P(x)$ is given by
\begin{equation}
P(x) = \sum_{i=0}^d b_i(x)P_i ,
\end{equation}
where $P_0, \dots, P_d \in S^n$; that is, the coefficients are symmetric matrices of the same dimensions. Because $P(x)$ is symmetric, it has real eigenvalues for every $x$. Before, we were concerned with nonnegativity of polynomials. The analogous property for polynomial matrices is positive semidefiniteness. We want conditions on the $P_i$ that guarantee that $P(x) \succeq 0$ for all $0 \leq x \leq 1$.

The approaches from the scalar case can be extended to the polynomial matrix case. We define the matrix analogues of $\mathcal{P}$ and $\mathcal{NB}$ as follows:
\begin{defn}
$\mathcal{P}_d^n$ is the set of $n \times n$ symmetric polynomial matrices $P(x)$ whose entries are in $\mathbb{R}_d[x]$ and $P(x) \succeq 0$ for all $0 \leq x \leq 1$.
\end{defn}
\begin{defn}
$\mathcal{NB}_d^n$ is the set of $n \times n$ symmetric polynomial matrices $P(x)$ whose entries are in $\mathbb{R}_d[x]$ and $P_i \succeq 0$ for every $0 \leq i \leq d$. 
\end{defn}

To define sets in the tradeoff space between these extremes, we generalize Lemma~\ref{lem:gen_nonnegative_socp} to the polynomial matrix case. Recall the definitions of $m_i$ and $w_i$ from \eqref{eqn:midef} and \eqref{eqn:widef}. 
\begin{restatable}{lem}{genmatrixnonnegativesocp}\label{lem:gen_matrix_nonnegative_socp} Let $P(x) = \sum_{i=0}^d b_i(x) P_i$ where $P_i \in S^n$. Let $C_0$ and $C_d$ be the zero matrix. If there exist $C_i \in \mathbb{R}^{n\times n}$ for $1 \leq i \leq d-1$ such that for all $1 \leq i \leq d-1$
    \begin{equation}\label{eqn:general_deg_matrix_pmp}
    \begin{aligned}
        \pmat{P_{i-1}-C_{i-1}-C_{i-1}^T & \sqrt{\frac{2m_i}{w_{i-1}w_{i+1}}}C_i \\ \sqrt{\frac{2m_i}{w_{i-1}w_{i+1}}} C_i^T & P_{i+1} - C_{i+1} - C_{i+1}^T} \succeq 0,
    \end{aligned}
    \end{equation}
    then $P\in \mathcal{P}_d^n$.
    \end{restatable}
\begin{proof} Let $\mathbf{1}_n$ be the $n \times n$ matrix of all 1's and $I_n$ be the $n \times n$ identity matrix. By Lemma~\ref{lemma:bernstein_socp}, 
    \begin{equation}
        0 \preceq B_i(x) := \pmat{b_{i-1}(x) & \sqrt{\frac{w_{i-1}w_{i+1}}{2m_i}}b_i(x)\\ \sqrt{\frac{w_{i-1}w_{i+1}}{2m_i}}b_i(x)& b_{i+1}(x)}
    \end{equation}
for every $0 \leq x \leq 1$. So on this domain, $B_i(x) \otimes \mathbf{1}_n \succeq 0$. Suppose the $C_i$ satisfy equation \eqref{eqn:general_deg_matrix_pmp} for $P$. The Schur product theorem implies the Hadamard product of $B_i(x) \otimes \mathbf{1}_n$ with the $i$th matrix in \eqref{eqn:general_deg_matrix_pmp} is positive semidefinite. Summing these Hadamard products over $1 \leq i \leq d-1$ gives 
\begin{equation}
    0 \preceq \pmat{\sum_{i=0}^{d-2} w_k(P_i - C_i - C_i^T) b_i(x) & \sum_{i=1}^{d-1} C_i b_i(x) \\ \sum_{i=1}^{d-1} C_i^T b_i(x) & \sum_{i=2}^{d} w_k(P_i - C_i - C_i^T) b_i(x)}.
\end{equation}
Finally, matrix congruence preserves positive semidefiniteness, so
\begin{equation}
    0 \preceq \pmat{I_n \\ I_n}^T \pmat{\sum_{i=0}^{d-2} w_k(P_i - C_i - C_i^T) b_i(x) & \sum_{i=1}^{d-1} C_i b_i(x) \\ \sum_{i=1}^{d-1} C_i^T b_i(x) & \sum_{i=2}^{d} w_k(P_i - C_i - C_i^T) b_i(x)} \pmat{I_n \\ I_n} = P(x)
\end{equation}
because the terms occurring twice have $w_i = \frac{1}{2}$ and the others have $w_i = 1$. Therefore $P \in \mathcal{P}_d^n$.
\end{proof}
Recall that in the scalar case our choice was given by~\eqref{eqn:arbitrary_ci}. This expression suggests that the choice for $C_i$ may require the matrix version of two natural operations: projection onto the ``positive'' part, and a generalized notion of geometric mean. 
Luckily, both of these are possible. As mentioned earlier, given a matrix $X$, its projection $X_+$ onto the positive semidefinite cone can be computed with an eigenvalue decomposition; details are presented in Appendix~\ref{sec:preliminaries-proofs}. The key property of the scalar geometric mean that we want to extend is: 
given $a,b \geq 0$, their geometric mean is the largest number $x \geq 0$ such that \begin{equation}
    \pmat{a & x \\ x & b} \succeq 0.
\end{equation} For the matrix case, we can ask for the generalized property: given $A, B \succeq 0$ (which will be $(P_{i-1})_+$ and $(P_{i+1})_+$), their ``geometric mean'' should be an $X \succeq 0$ such that 
\begin{equation}\label{eqn:geomeankeyfact}
    \pmat{A & X \\ X & B} \succeq 0 .
    \end{equation}
It turns out that the \emph{matrix geometric mean} of $A$ and $B$ has exactly this property. 
\begin{defn}
    Let $A, B \succ 0$. The \emph{geometric mean} of $A$ and $B$, denoted $A\#B$, is given by 
    \begin{equation}
    A\#B = A^\frac{1}{2}(A^{-\frac{1}{2}}B A^{-\frac{1}{2}})^{\frac{1}{2}}A^{\frac{1}{2}}.
    \label{eq:geomean}
    \end{equation}
    One can extend the definition to all $A,B \succeq 0$ by continuity.
    \label{def:geomean}
\end{defn}
\noindent Although~\eqref{eq:geomean} is not obviously symmetric in $A$ and $B$, it indeed holds that $A\#B = B\#A$. There are many other beautiful properties of the matrix geometric mean, including strong connections to Riemannian geometry; see e.g. \cite{bhatia2009positive, hiai2014introduction, bhatia2006riemannian}. The following fact 
is the key property we need. 
\begin{fact}[Theorem 3.4 of \cite{lawson2001geometric}]\label{fact:key-geom-mean}The geometric mean $A \# B$ is the largest (in the Loewner order) $X \in S^n$ such that \eqref{eqn:geomeankeyfact} holds.
\end{fact}

With the correct notions of positive projection and geometric mean in place, we can now define $\mathcal{GB}_d^n$ as a natural generalization of $\mathcal{GB}_d$.
Let $\mathbf{0}_n$ be the $n\times n$ zero matrix. Let $P_{i} = \mathbf{0}_n$ if $i<0$ or $i > d$. We choose  
\begin{equation}\label{eqn:cichoicematrix}
    C_i = -\sqrt{\frac{w_{i-1}w_{i+1}}{2m_i}}((P_{i-1})_+\#(P_{i+1})_+),
\end{equation}
which appropriately generalizes \eqref{eqn:cubic_c1c2} and \eqref{eqn:scalar-ci}.
\begin{defn}
     $\mathcal{GB}_d^n$ is the set of $n\times n$ polynomial matrices $\sum_{i=0}^db_i(x)P_i$ such that for all $0 \leq i \leq d$, \begin{equation}\label{eqn:GBdn} P_i \succeq -\sqrt{\frac{2w_{i-1}w_{i+1}}{{m_i}}}\left((P_{i-1})_+ \# (P_{i+1})_+\right).\end{equation}
     \label{def:GBmatrix}
\end{defn}

Finally, we show this definition is in the trade space between $\mathcal{NB}_d^n$ and $\mathcal{P}_d^n$, and this theorem implies Theorems~\ref{thm:1in2in3} and \ref{thm:1in2in3_general}.
\begin{restatable}{thm}{arbitrarygeneralinclusions}\label{thm:arbitrary_general_inclusions}
    $\mathcal{NB}_d^n \subset \mathcal{GB}_d^n \subset \mathcal{P}_d^n$.\end{restatable}
    \begin{proof}
        First suppose that $P \in \mathcal{NB}_d^n$. Since $P_i \succeq 0$, if $M \in S^n$ has nonpositive eigenvalues, then $P_i - M \succeq 0$. In particular, $-1$ times the geometric mean of $P_{i-1}$ and $P_{i+1}$ has nonpositive eigenvalues. Therefore $P \in \mathcal{GB}_d^n$. 
        
        Next we argue that if $P \in \mathcal{GB}_d^n$ then choosing $C_i$ as in \eqref{eqn:cichoicematrix} is feasible for \eqref{eqn:general_deg_matrix_pmp}. If either $P_{i-1}$ or $P_{i+1}$ is indefinite, then the only thing to check is that both $P_{i-1} - C_{i-1} - C_{i-1}^T$ and $P_{i+1} - C_{i+1} - C_{i+1}^T$ are positive semidefinite (because the off-diagonal entries are 0). This is exactly the condition that $P \in \mathcal{GB}$. On the other hand, if both are positive semidefinite, then we need to show
        \begin{equation}\label{eqn:GBdecomp}
        \pmat{P_{i-1}  & \sqrt{\frac{2m_i}{w_{i-1}w_{i+1}}} C_i \\ \sqrt{\frac{2m_i}{w_{i-1}w_{i+1}}}C_i^T & P_{i+1}} + \pmat{- C_{i-1} - C_{i-1}^T  & 0 \\ 0 & - C_{i+1} - C_{i+1}^T} \succeq 0.
        \end{equation} 
        By Fact~\ref{fact:key-geom-mean}, \begin{equation} \pmat{P_{i-1}  & -\sqrt{\frac{2m_i}{w_{i-1}w_{i+1}}}C_i \\ -\sqrt{\frac{2m_i}{w_{i-1}w_{i+1}}}C_i^T &P_{i+1}} \succeq 0.
        \end{equation} This means for all $x, y \in \mathbb{R}^n$, 
        \begin{equation}\begin{aligned}
        0 &\leq \pmat{x \\y}^T\pmat{P_{i-1}  & -\sqrt{\frac{2m_i}{w_{i-1}w_{i+1}}}C_i \\ -\sqrt{\frac{2m_i}{w_{i-1}w_{i+1}}}C_i^T &P_{i+1}}\pmat{x \\y}\\ &= \pmat{x \\-y}^T\pmat{P_{i-1}  & \sqrt{\frac{2m_i}{w_{i-1}w_{i+1}}}C_i \\ \sqrt{\frac{2m_i}{w_{i-1}w_{i+1}}}C_i^T &P_{i+1}}\pmat{x \\-y},
        \end{aligned}
        \end{equation}
        so the first matrix in \eqref{eqn:GBdecomp} is positive semidefinite. The second matrix is positive semidefinite because the block diagonals are geometric means, which are positive semidefinite. Therefore the sum is as well. Hence $P \in \mathcal{P}_d^n$. 
        \end{proof}

\section{Numerical experiments}
\label{sec:numerical}
For every application that membership in $\mathcal{NB}$ is used as a sufficient condition for nonnegativity, 
it could be worth considering how testing membership in $\mathcal{GB}$ instead may provide a benefit. Our examples deal with certifying a lower bound $\delta$ on a polynomial $p$, or proving $p - \delta$ is nonnegative. We compare how well $\mathcal{NB}$ and $\mathcal{GB}$ work as termination criteria for the subdivision method. 

Recall the subdivision method described in Section~\ref{sec:preliminaries}. The computation required by De Casteljau's algorithm in Step 2 exceeds the cost of checking $p \in S$ in Step 1 when $S = \mathcal{NB}$ or $\mathcal{GB}$, so saving on the number of subdivisions can provide a significant advantage. We therefore count the number of subdivisions required when using the two different nonnegativity criteria. 

\subsection{Varying root locations}
The first experiment considers the quadratic polynomials $(x-t)^2$, written in the cubic Bernstein basis. The rationale of using such a simple model is to gain a fundamental understanding of how the locations of roots affect the subdivision method. Generically, polynomials are locally quadratic near local minima, thus by studying subdivisions of quadratics one can hope to learn about the behavior of higher degree polynomials with several local minima.

Given a small fixed value of $\delta>0$, in Figure~\ref{fig:numsubsquad} we report the number of subdivisions required to prove $(x-t)^2\geq -\delta$. The fractal pattern is related to the binary expansion of the root $t$. Both methods do better when a subdivision occurs near the minimizer. For example, $(x- \frac{1}{2})^2$ only takes one subdivision because when we subdivide $[0,1]$, the first subdivision split is at $x = \frac12$. When $\mathcal{NB}$ is the termination criterion, a small change in the location of the minimizer may have a big effect on the number of subdivisions required. On the other hand, $\mathcal{GB}$ is less sensitive to the exact location of the minimizer. We can observe from Figure~\ref{fig:numsubsquad} that the number of subdivisions required with $\mathcal{GB}$ is never bigger than the number required with $\mathcal{NB}$.

\medskip

\begin{figure}[htbp]
    \centering
    \resizebox{\textwidth}{!}{\begin{tikzpicture}
    \begin{axis}[
      xlabel=$t$,
      ylabel=Number of subdivisions,
      ymin = -.2,
      grid = both,
      xmin = -.01, xmax = 1.01,
      width = \textwidth, 
      height = .5\textwidth, 
      legend style={at={(0.5,0.0)},anchor=south,
                  font=\footnotesize,
                  },
                  ]
    \addplot[color = blue] table[x=ts, y=nsubs, col sep=comma] {sections/data/numsubsquadratic1e-4.csv};
    \addlegendentry{$\mathcal{NB}$}
    \addplot[color = red, dashed, thick] table[x=ts, y=gsubs, col sep=comma] {sections/data/numsubsquadratic1e-4.csv};
    \addlegendentry{$\mathcal{GB}$}
    \end{axis}
    \end{tikzpicture}}\vspace{-.5cm}
    \caption{Number of subdivisions required to prove $(x-t)^2 \geq -\delta$ for $\delta = 10^{-4}$ with the subdivision method and different nonnegativity criteria. When $\mathcal{GB}$ is the nonnegativity criterion, the number of subdivisions is less sensitive to whether a subdivision location exactly coincides with $t$.}
    \label{fig:numsubsquad}
\end{figure}
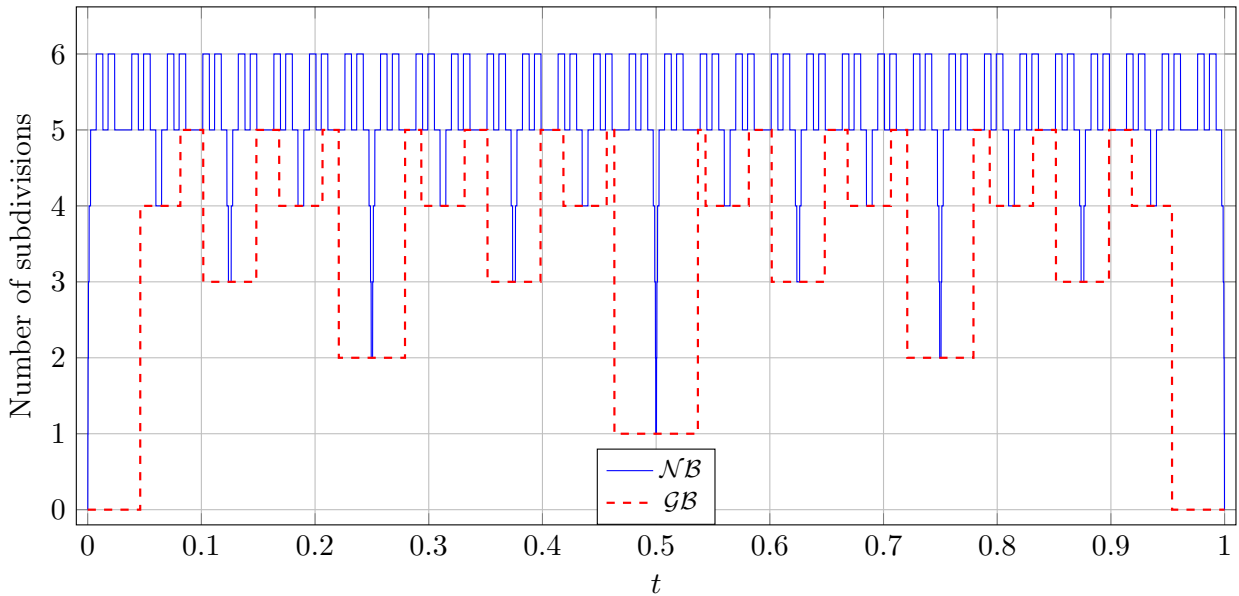

The same phenomenon is presented in a different way in Figure~\ref{fig:avgsubsquad}. We show the percentage of roots $t$ on $[0,1]$ for which fewer than $1,2,3,4,5,$ and 6 subdivisions are required with each of the nonnegativity criteria to prove $(x-t)^2 \geq -\delta$ for a small $\delta > 0$. 
We see that a nontrivial portion of the ensemble requires just a couple subdivisions when we check nonnegativity with $\mathcal{GB}$, whereas checking with $\mathcal{NB}$ frequently requires the maximum number of subdivisions.

\smallskip
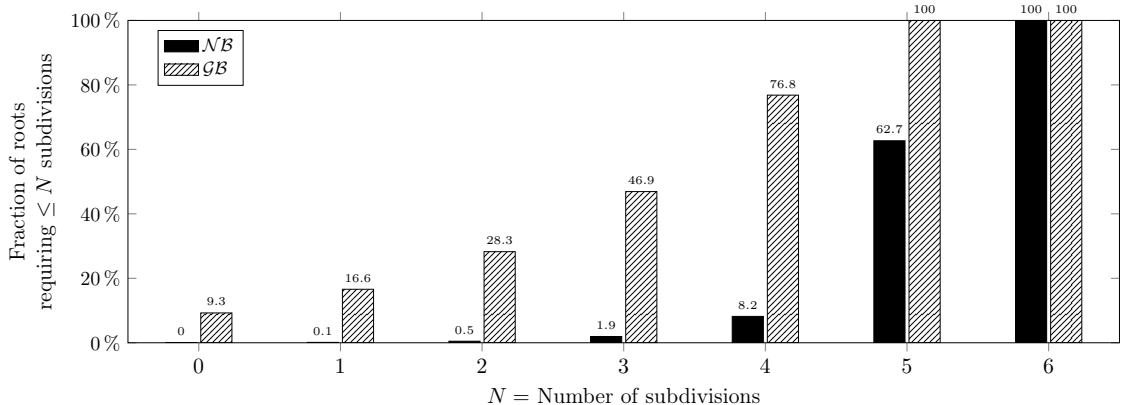
\begin{figure}[htbp]
    \centering
    \resizebox{.9\textwidth}{!}{\begin{tikzpicture}
    \begin{axis}[
        bar width=0.22,
        ybar,
        x = 2.5 cm,
        xtick = data,
        enlarge x limits={abs=0.5},
        ymin=0,
        ymax=100,
        ylabel style={align=center},
        ytick distance=20,
        yticklabel style={font=\small},
        yticklabel=\pgfmathprintnumber{\tick}\,$\%$,
        nodes near coords,
        nodes near coords style={font=\tiny,
                                 /pgf/number format/.cd,
                                 fixed,
                                 precision=1,
                                 },
        legend style={legend pos=north west,
                      cells={anchor=west},
                      font=\footnotesize,
                      },
      xlabel= {$N = $ Number of subdivisions},
      ylabel={Fraction of roots\\ requiring $\leq N$ subdivisions},
      ymin = 0,
      area legend]

    \addplot[fill = black] table[x=numsubdivs,y=NB] {subdiv-avg.dat};
    \addlegendentry{$\mathcal{NB}$}
    \addplot[pattern=north east lines] table[x=numsubdivs,y=GB] {subdiv-avg.dat};
    \addlegendentry{$\mathcal{GB}$}
    \end{axis}
    \end{tikzpicture}}
    \caption{Percentage of instances for which the subdivision method requires no more than $N$ subdivisions to prove $(x-t)^2 \geq -\delta$  (with $\delta = 10^{-4}$, and $t$ uniformly distributed in $[0,1]$). Even with just a few subdivisions, a large number of the bounds can be proven with the subdivision method and $\mathcal{GB}$ as a nonnegativity condition.}
    \label{fig:avgsubsquad}
\end{figure}

\subsection{Varying matrix size}
The second experiment compares different termination criteria in the subdivision method for matrix-valued polynomials of varying sizes whose eigenvalues approach 0 for many $x$. In this case, the subdivisions $P^1(x)$ and $P^2(x)$ are defined so that
\begin{equation}\label{eqn:subdivision-def-matrix}
    P(x) = \begin{cases} P^1(2x) \qquad  &0 \leq x \leq \frac{1}{2}\\
     P^2(2x-1) \qquad &\frac{1}{2} \leq x \leq 1.\end{cases}
\end{equation}
For each $n$, we sample 100 $n\times n$ matrices of the form
\begin{equation}\label{eqn:random_matrix_model}
    T^T \pmat{\rho_1(x) & 0 & \cdots & 0\\
                0 & \rho_2(x) & \cdots & 0\\
                0 & 0 & \ddots & 0\\
                0 & 0 & \cdots & \rho_n(x)} T,
\end{equation}
where the entries of $T$ are independently and identically distributed Gaussian random variables and all $\rho(x)$ are random cubic polynomials nonnegative on the interval. The random congruence transformation means that \eqref{eqn:random_matrix_model} is still always positive semidefinite but implies the eigenvalues are not polynomials in $x$. An example of the eigenvalues for one such $10 \times 10$ matrix are given in Figure~\ref{fig:example_eigs_10x10}.
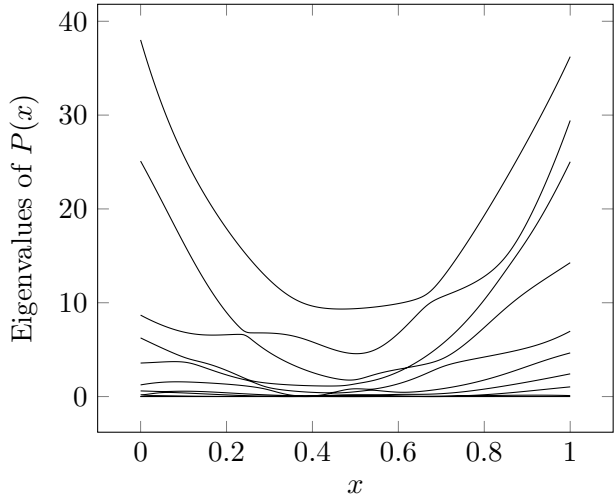
\begin{figure}[htbp]
    \centering
    \begin{tikzpicture}
        \begin{axis}[
          xlabel=$x$,
          ylabel=Eigenvalues of $P(x)$]
        \addplot+[smooth, black, mark = none] table {sections/data/random_eigs_10x10.dat};
        \end{axis}
        \end{tikzpicture}
        \caption{Eigenvalues of a random $10 \times 10$ polynomial matrix sampled according to \eqref{eqn:random_matrix_model}. Notice that the eigenvalues are small for many $x$.}
    \label{fig:example_eigs_10x10}
\end{figure}

We run the subdivision algorithm on the matrix polynomials to prove $P(x) \succeq -\delta I$ using $\mathcal{NB}$ and $\mathcal{GB}$ as nonnegativity criteria. For each matrix we compute the difference in the number of subdivisions required with each criteria. The average and standard deviations across all 100 matrices are plotted in Figure~\ref{fig:matrix_subdiv_demo2}. As the matrices get larger, the average number of subdivisions saved with $\mathcal{GB}$ increases. 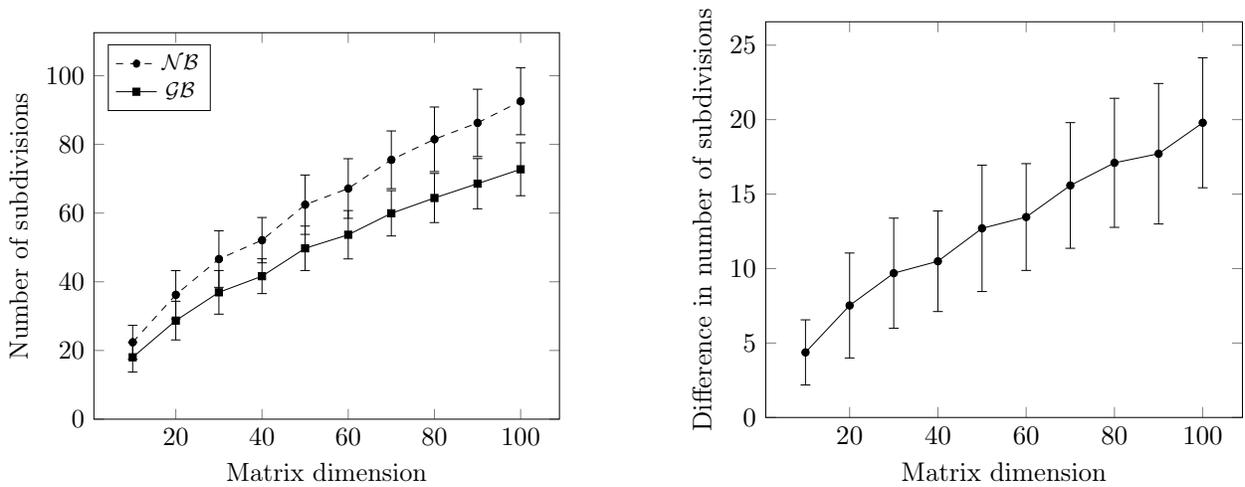
\begin{figure}[htbp]
    \centering\begin{subfigure}{.45\textwidth}
    \resizebox{\textwidth}{!}{\begin{tikzpicture}
        \begin{axis}[
          xlabel=Matrix dimension,
          ylabel=Number of subdivisions,
          ymin = 0,
          legend style={legend pos=north west,
                      font=\footnotesize,
                      },]

        \addplot+[forget plot, mark options={black, scale=0.75}, black, only marks, error bars/.cd, 
        y fixed,
        y dir=both, 
        y explicit] table[x=N,y=NB:,y error=NBstd] {sections/data/matrix_subdiv4.dat};
        \addplot+[mark options={black, scale=0.75}, black, dashed, error bars/.cd, 
        y fixed,
        y dir=both, 
        y explicit] table[x=N,y=NB:] {sections/data/matrix_subdiv4.dat};

        \addlegendentry{$\mathcal{NB}$}
        \addplot+[mark options={black, scale=0.75}, black, error bars/.cd, 
        y fixed,
        y dir=both, 
        y explicit] table[x=N,y=GB:,y error=GBstd] {sections/data/matrix_subdiv4.dat};
        \addlegendentry{$\mathcal{GB}$}

        \end{axis}

        \end{tikzpicture}}
\end{subfigure} \hfill
    \begin{subfigure}{.45\textwidth}
        \resizebox{\textwidth}{!}{\begin{tikzpicture}
      \begin{axis}[
        xlabel=Matrix dimension,
        ylabel=Difference in number of subdivisions,
        ymin = 0,
        legend style={legend pos=north west,
                    font=\footnotesize,
                    },]

      \addplot+[mark options={black, scale=0.75}, black, error bars/.cd, 
      y fixed,
      y dir=both, 
      y explicit] table[x=N,y=Diff,y error=Diffstd] {sections/data/matrix_subdiv4.dat};
      \end{axis}

      \end{tikzpicture}}\end{subfigure}
    \caption{Number of subdivisions required with $\mathcal{NB}$ or $\mathcal{GB}$ as nonnegativity criteria to prove $P(x) \succeq \delta I$ with $\delta = 10^{-4}$ for 100 random matrices sampled according to \eqref{eqn:random_matrix_model}. The average number of subdivisions saved with $\mathcal{GB}$ instead of $\mathcal{NB}$ increases as the matrices get bigger. The percentage of subdivisions saved is roughly 20\%. Vertical bars are one standard deviation of counts (left) or differences (right).}
    \label{fig:matrix_subdiv_demo2}
\end{figure}

\section{Conclusion}
We developed novel simple and explicit conditions to certify nonnegativity of Bernstein polynomials. The new tests better balance the tradeoffs between exact but expensive conditions, and the commonly used test based on nonnegative Bernstein coefficients. The method is based on making explicit choices for the decision variables in the SDP/SOCP characterizations of nonnegativity, bypassing the need to solve them numerically.

There are several related open areas for potential further work. An open question is whether there are other reasonable low-complexity choices for the decision variables (that may violate the hypotheses of Proposition~\ref{prop:tight} or that may not satisfy the conditions of Theorem~\ref{thm:1in2in3}). Generalizing the basic idea of Proposition~\ref{prop:tight} to higher degrees and the polynomial matrix case is also future work. Finally, it would be interesting to do a more comprehensive evaluation of how well these techniques perform in different applied settings.

\bibliography{references}
\bibliographystyle{plain}
\appendix
\section{Projections, square roots, and geometric means }\label{sec:preliminaries-proofs}

The orthogonal projection $X_+$ of a symmetric matrix $X$ onto the positive semidefinite cone is defined as
\begin{equation}
   X_+ =  \underset{{Z \succeq 0}}{\textrm{argmin}}\, \Vert X - Z \Vert_\text{Fro}^2,
\end{equation}
where $\Vert X \Vert_\text{Fro}$ is the Frobenius norm of $X$. 
This projection can be computed with the standard algorithm: 
\begin{enumerate}
    \item Find an eigendecomposition of $X = PDP^{T}$ with $D$ diagonal.
    \item Define $(D_+)_{ij} = (D_{ij})_+$ for every $1 \leq i,j \leq n$.
    \item Return $X_+ = PD_+P^{T}$.
\end{enumerate}

To compute the matrix geometric mean using  Definition~\ref{def:geomean}, we need the matrix square root. Recall that the square root of a diagonal matrix is the diagonal matrix whose entries are the square roots of the original matrix. Then the matrix square root of any positive semidefinite matrix $X^{\frac{1}{2}}$ can be computed as $PD^{\frac{1}{2}}P^{T}$ where $PDP^{T}$ is an eigendecomposition of $X$. We also remark that there are other, more efficient algorithms to compute the matrix geometric mean; see e.g.~\cite{Iannazzo}.

\section{Alternative approach to Lemma~\ref{lem:gen_nonnegative_socp}}\label{sec:altsosproof}

We note in Section~\ref{sec:generalization} that there is an alternative proof of Lemma~\ref{lem:gen_nonnegative_socp}. We sketch this proof here.

\gennonnegativesocp*
\begin{proof}[Alternative proof sketch of Lemma~\ref{lem:gen_nonnegative_socp}] Suppose the degree of $p$ is $2d+1$. Let $M_1$ and $M_2$ be tridiagonal matrices with 
    \begin{equation}
    \begin{aligned}
        M_1 = \pmat{
        a_0p_0 & f_1c_1 & 0 & 0 & \cdots  & 0\\
        f_1c_1 & a_2(p_2 - c_2) & f_3c_3 & 0& \cdots &0\\
        0 & f_3c_3 & a_4(p_4 - c_4) & f_5c_5 & \cdots & 0\\
        \vdots & \vdots & \vdots & \vdots & \ddots & \vdots \\
        0 & 0 & 0 & 0 & \cdots & a_{2d}(p_{2d} - c_{2d})}, \\M_2 = \pmat{
        a_1(p_1-c_1) & f_2c_2 & 0 & 0 & \cdots  & 0\\
        f_2c_2 & a_3(p_3 - c_3) & f_4c_4 & 0& \cdots &0\\
        0 & f_4c_4 & a_5(p_5 - c_5) & f_6c_6 & \cdots & 0\\
        \vdots & \vdots & \vdots & \vdots & \ddots & \vdots \\
        0 & 0 & 0 & 0 & \cdots & a_{2d+1} p_{2d+1}},
        \end{aligned}
    \end{equation}
    where $a_{2k} = \frac{\binom{2d+1}{2k}}{\binom{d}{k}^2}$, $a_{2k+1} = \frac{\binom{2d+1}{2k+1}}{\binom{d}{k}^2}$, $f_{2k+2} = \frac{\binom{2d+1}{2k+2}}{2\binom{d}{k}\binom{d}{k+1}}, f_{2k+1} = \frac{\binom{2d+1}{2k+1}}{2\binom{d}{k}\binom{d}{k+1}}$.
    Let $\vec{b} = [b_0(x), \dots, b_d(x)]^T$. Then one can show
    \begin{equation}\label{eqn:generalmarkovlukacs}
       p = (1-x) \vec{b}^{T} M_1 \vec{b} +x \vec{b}^{T} M_2 \vec{b} .
    \end{equation}
Therefore, if $M_1$ and $M_2$ are positive semidefinite, then $p$ is nonnegative on the interval. If \eqref{eqn:general-scalar-sufficient} is feasible, then $M_1, M_2\succeq 0$, because they can be constructed by interlacing positive semidefinite $2 \times 2 $ blocks. Interlacing $2 \times 2$ matrices here means to form a tridiagonal matrix by adding the $(1,1)$ entry of each subsequent $2 \times 2$ matrix to the $(2,2)$ entry of the preceding $2 \times 2$ matrix along the diagonal. What remains is to show \eqref{eqn:general-scalar-sufficient} implies the required $2 \times 2$ matrices are positive semidefinite. The case when $d$ is even is similar. In both cases, the condition is equivalent to certain tridiagonal matrices being positive semidefinite.
    \end{proof}
\section{Maximality of \texorpdfstring{$\mathcal{GB}$}{GB}}\label{sec:p2proofs}

This section contains the proof of the following proposition.

\tight*
\begin{proof}[Proof of Proposition~\ref{prop:tight}]

We enumerate the inequalities equivalent to the second order cone conditions  \eqref{eqn:nonnegative_socp} so that we can refer to each of them individually. They are:
\begin{align}
    4p_0(3p_2-c_2) &\geq c_1^2 \label{eqn:p1constraintquad}\\ 3p_2-c_2 &\geq 0 \label{eqn:p1constraintlin}\\
    4p_3(3p_1-c_1) &\geq c_2^2\label{eqn:p2constraintquad} \\3p_1 - c_1 &\geq 0 \label{eqn:p2constraintlin} .
\end{align}
For this proof, $c_1$ and $c_2$ are explicit functions of $p$. We let $c_1 = f(p_0,p_2)$, $c_1^*=f^*(p_0,p_2)$, $c_2 = f(p_3,p_1)$ and $c_2^* = f^*(p_3,p_1)$.

The proof structure is as follows. Let $p \in \mathcal{GB} \setminus \mathcal{NB}$ for which $c_1 \neq c_1^*$ and $c_2 \neq c_2^*$. 
If $p \not \in S(f)$, we are done, because we have a point in $S(f^*)$ not in $S(f)$. First, we consider $c_2 > c_2^*$ and produce a point $q \in S(f^*)$ and $q \not \in S(f)$. A symmetric argument produces such a $q$ if $c_1> c_1^*$. Second, we assume that $c_1 < c_1^*$ and $c_2 < c_2^*$, and then show the existence of a sequence $\{q^k\}_{k=1}^\infty$ for which eventually $q^k \in S(f^*)$ but $q^k \not \in S(f)$. 
If $p \in \mathcal{GB} \setminus \mathcal{NB}$, then exactly one of $p_1$ and $p_2$ is negative. We treat these possibilities as two subcases.
\paragraph{$\mathbf{c_2 > c_2^*:}$} 
\begin{itemize}
\item[$p_2 < 0$] Let $q = (p_0, p_1, \frac{1}{3}c_2^*, p_3)$. 
The value $c_2 = f(q_3, q_1) = f(p_3, p_1) $ remains unchanged. Constraint \eqref{eqn:p1constraintlin} requires that $3q_2 - f(q_3, q_1) = c_2^* - c_2 \geq 0$, which is violated because $c_2 > c_2^*$. Therefore $q \not \in S(f)$. On the other hand, $3q_2 - f^*(q_3, q_1) = c_2^* - c_2^* = 0$. Furthermore $f^*(q_0,q_2) = 0$ (since $c_2^* \leq 0$), so \eqref{eqn:p1constraintquad} is satisfied too. Because $c_1^* = 0$ and $f^*(p_3,p_1) = f^*(q_3,q_1)$, \eqref{eqn:p2constraintquad} and \eqref{eqn:p2constraintlin} are still satisfied for $q$, so $q \in S(f^*)$.

\item[$p_2 > 0$] Let $q = (\frac{3p_1^2}{4p_2}, p_1, p_2, p_3)$. 
When $p_2 >0$, $p_1 <0$, so $f^*(p_1,p_3) = 0$. Therefore $f(p_1,p_3) > 0$. Suppose for contradiction $q \in S(f)$. If \eqref{eqn:p1constraintquad} were satisfied for $q$, then $9p_1^2 > c_1^2$. Furthermore, \eqref{eqn:p2constraintlin} would imply $3p_1 - c_1 \geq 0$, or $3p_1 \geq c_1$. Since $p_1 < 0$, this implies $9p_1^2 \leq c_1^2$. This contradicts what we showed above. Therefore $q \not \in S(f)$. However $f^*(q_3,q_1) = 0$ and $f^*(q_0, q_2) = 3p_1$, satisfies all the constraints, so $q \in S(f^*)$. 
\end{itemize}

\paragraph{$\mathbf{c_2< c_2^*:}$}  Suppose $f(p_3, p_1) < f^*(p_3, p_1)$ and $f(p_0,p_2) < f^*(p_0,p_2)$. Without loss of generality suppose $p_2 < 0$. 
Let $q^k = (\frac{1}{k}, p_1, \frac{1}{3}f^*(p_1,p_3), p_3)$. 

We argue that for some sufficiently large $k$, $q^k \not \in S(f)$. Notice that $f(q_3,q_1)$ does not change as $k$ increases. If $q^k \in S(f)$ then \eqref{eqn:p1constraintquad} implies $4q_0(3q_2 - f(p_3,p_1)) \geq f(p_0,p_2)^2$, or substituting, $\frac{4}{k}(f^*(p_3,p_1) - f(p_3,p_1)) \geq f(\frac{1}{k}, \frac{1}{3}f^*(p_3,p_1))^2$, so 
\begin{equation}\label{eqn:tightlb}f\Big(\frac{1}{k}, \frac{1}{3}f^*(p_3,p_1)\Big) \geq -\sqrt{\frac{4}{k}(f^*(p_3,p_1) - f(p_3,p_1))}. \end{equation}
At the same time \eqref{eqn:p2constraintquad} would imply
$c_1 \leq 3q_1 - \frac{c_2^2}{4q_3}$, or \begin{equation}\label{eqn:tightub}f\Big(\frac{1}{k}, \frac{1}{3}f^*(p_3,p_1)\Big) \leq 3p_1 - \frac{1}{4p_3}f(p_3,p_1)^2 < 0.\end{equation}
The second inequality holds because $f(p_3, p_1) < f^*(p_3, p_1)$, and the middle expression vanishes when substituting $f^*(p_3,p_1)$ for $f(p_3,p_1)$. For some large enough $k$, the lower bound \eqref{eqn:tightlb} approaches zero and will exceed the constant nonpositive upper bound provided by \eqref{eqn:tightub}.

Finally, we show that for every $k$, $q^k \in S(f^*)$. As $f^*(p_3,p_1) \leq 0$, for every $k$, $f^*(q^k_0,q^k_2) = 0$. Therefore \eqref{eqn:p1constraintquad} and \eqref{eqn:p1constraintlin} are satisfied since $\frac{1}{k} > 0$ and $3(\frac{1}{3}f^*(p_3,p_1)) - f^*(p_3,p_1) = 0$. If $p \in S(f^*)$, then \eqref{eqn:p2constraintquad} and \eqref{eqn:p2constraintlin} must be satisfied for each $q^k$ as well because nothing in those conditions has changed. Hence $q^k \in S(f^*)$.
\end{proof}

\section{Quantifier elimination}\label{sec:cubic-exact-proof}

Theorem~\ref{thm:exactcubic} can be verified automatically using decision algebra. 
Specifically, quantifier elimination algorithms are systematic procedures to eliminate quantifiers in first-order formulas over the real field; see e.g.~\cite{CavinessJohnson}. A well-known implementation is \texttt{QEPCAD} \cite{brown2003qepcad}. We demonstrate how to use \texttt{QEPCAD} to eliminate the quantifier in the logical expression defining $p \in \mathcal{P}^\circ$,
\begin{equation}
    \forall{x} 
    (((0\leq x) \land (x \leq 1)) \Longrightarrow p_0(1-x)^3 + 3 p_1x(1-x)^2 + 3p_2x^2(1-x) + p_3x^3 > 0) ,
  \end{equation}
  to get 
  
  \begin{equation}p_0 >0 \land p_3 > 0 \land (D(p) < 0 \lor (p_1 > 0 \land p_2>0)).
  \end{equation}

\smallskip

  
   \begin{tiny} 
   \begin{lstlisting}[multicols=2, caption={Verification of Theorem~\ref{thm:exactcubic}}]
    $ ./qepcad +N80000000
=======================================================
                Quantifier Elimination                 
                          in                           
            Elementary Algebra and Geometry            
                          by                           
      Partial Cylindrical Algebraic Decomposition      
                                                       
               Version B 1.65, 10 May 2011
                                                       
                          by                           
                       Hoon Hong                       
                  (hhong@math.ncsu.edu)                
                                                       
With contributions by: Christopher W. Brown, George E. 
Collins, Mark J. Encarnacion, Jeremy R. Johnson        
Werner Krandick, Richard Liska, Scott McCallum,        
Nicolas Robidoux, and Stanly Steinberg                 
=======================================================
Enter an informal description  between '[' and ']':
[Eliminate the quantifiers in P^circ description]
Enter a variable list:
(p0,p1,p2,p3,x)
Enter the number of free variables:
4
Enter a prenex formula:
(A x)[ [0 <= x /\ x <= 1] ==> (1-x)^3 p0 + 3 p1 (1-x)^2 x + 3 p2 (1-x) x^2 + p3 x^3 > 0 ].


=======================================================

Before Normalization >
go

Before Projection (x) >
proj-op (m,m,h,h)

Before Projection (x) >
finish

An equivalent quantifier-free formula:

p0 > 0 /\ p3 > 0 /\ [ p0^2 p3^2 - 6 p0 p1 p2 p3 + 4 p1^3 p3 + 4 p0 p2^3 - 3 p1^2 p2^2 > 0 \/ [ p1 > 0 /\ p2 > 0 ] ]


=====================  The End  =======================

--------------------------------------------------- --------------------------
0 Garbage collections, 0 Cells and 0 Arrays reclaimed, in 0 milliseconds.
16973738 Cells in AVAIL, 40000000 Cells in SPACE.

System time: 1453 milliseconds.
System time after the initialization: 1268 milliseconds.
---------------------------------------------------- -------------------------
     \end{lstlisting}
    \end{tiny}

  
    The projection option we use tells \texttt{QEPCAD} to perform Hong's projection algorithm \cite{hong1990improvement}. The default projection algorithm in \texttt{QEPCAD} sometimes gives warnings even though they can be safely ignored, which is the case with our formula. 

\end{document}